\documentclass[10pt,a4paper]{article}
\usepackage[utf8]{inputenc}
\usepackage[english]{babel}
\usepackage[dvipsnames]{xcolor}
\usepackage{amsmath}
\usepackage{amsfonts}
\usepackage{amsthm}
\usepackage{amssymb}
\usepackage{enumerate}
\usepackage{tikz}
\usepackage{graphicx}
\usepackage[margin=1.2in]{geometry}
\author{Heidi Van den Camp}

\title{The Effect of Symmetry-preserving Operations on 3-Connectivity}
\date{{\small Ghent University\\ Department of Applied Mathematics, Computer Science, and Statistics\\ Krijgslaan 281 S9\\ 9000 Ghent, Belgium}}

\newtheorem{theorem}{Theorem}

\newtheorem{definition}{Definition}
\newtheorem{lemma}{Lemma}
\newtheorem{corollary}{Corollary}
\newtheorem*{theorem*}{Theorem}
\newtheorem*{corollary*}{Corollary}

\newtheorem*{remark*}{Remark}

\begin{document}
	\maketitle
	\section*{Abstract}
	In 2017, Brinkmann, Goetschalckx and Schein introduced a very general way of describing operations on embedded graphs that preserve all orientation-preserving symmetries of the graph. This description includes all well-known operations such as Dual, Truncation and Ambo. As these operations are applied locally, they are called local orientation-preserving symmetry-preserving operations (lopsp-operations). In this text we will use the general description of these operations to determine their effect on 3-connectivity. Recently it was proved that all lopsp-operations preserve 3-connectivity of graphs that have face-width at least three. We present a simple condition that characterises exactly which lopsp-operations preserve 3-connectivity for all embedded graphs, even for those with face-width less than three. 
	
	\section{Introduction}	
		
	Polyhedra have fascinated people for a very long time. Through the centuries, many renowned scientists such as Kepler, Coxeter \cite{coxeter1973regular, coxeter1954uniform}, and Conway \cite{conway2016symmetries} have worked with these fascinating objects. During this time numerous symmetry-preserving operations on polyhedra have been studied. The most well-known operations are probably \emph{Dual}, \emph{Medial} and \emph{Truncation}, but there are many others. In chemistry for example operations such as \emph{Leapfrog}, \emph{Chamfer} and \emph{Capra} are important in the study of fullerenes \cite{diudea2006generalized, fowler1992symmetry, king2006chirality}. These symmetry-preserving operations stem from the study of polyhedra i.e.\ 3-connected plane graphs, but they are also applied to more general embedded graphs \cite{diudea2003leapfrog,pisanski2017operations}. In that last article several, but not all, operations are described by their effect on the \emph{arcs} of an embedded graph.
	An obvious extension to higher genus embeddings is to apply the operations to polyhedral --that is, locally plane-- embeddings, which are a generalisation of polyhedra to higher genus, but it is perfectly possible to define the application of such operations for embeddings that are not locally plane. Polyhedral embeddings often make up only a very small fraction of all embedded graphs, especially in higher genera \cite{vcleemput2021genera}. For example, of all graphs on 28 vertices only 0,34\% have a polyhedral embedding. This shows the importance of also considering embeddings that are not polyhedral.

	In \cite{brinkmann2017comparing} Brinkmann, Goetschalckx and Schein presented a general definition to describe symmetry-preserving operations on embedded graphs. This approach was further generalised and discussed in more detail in \cite{brinkmann2021local}. In \cite{brinkmann2017comparing} the focus is on operations that preserve all symmetries, so-called local symmetry-preserving or lsp-operations, and in \cite{brinkmann2021local} the focus is on the more general class of operations that are only guaranteed to preserve the orientation-preserving symmetries, called local orientation-preserving symmetry-preserving operations or lopsp-operations. Examples of well-known operations that are lopsp-operations but not lsp-operations are Snub, Gyro, Capra, and the chiral Goldberg-Coxeter operations \cite{dutour2004goldberg}. With these definitions research is no longer limited to the specific operations that have been described in the past. Furthermore, it becomes possible to prove general statements about all lsp- or all lopsp-operations at the same time.
	
	In this paper, we look at the effect of lopsp-operations on 3-connectivity of simple embedded graphs of arbitrary genus. It was proved in \cite{brinkmann2017comparing} that all lsp-operations associated with 3-connected tilings preserve 3-connectivity of plane embedded graphs. In \cite{brinkmann2021local} this was proved for lopsp-operations applied to polyhedral embeddings, which is a much more general result. However, it is known that there exist simple embedded 3-connected graphs that have a dual which is simple and not 3-connected. In fact, it was proved by Bokal, Brinkmann, and Zamfirescu that for any $k\geq 1$ there exist $k$-connected graphs with a simple dual that has a 1-cut \cite{bokal2022connectivity}. As Dual is a lopsp-operation --even an lsp-operation-- it follows that not all lopsp-operations always preserve 3-connectivity.
	In this text we prove that Dual  is not the only lopsp-operation with this property: We define the class of edge-breaking operations using an easy condition and prove that these are exactly the operations that do not always preserve 3-connectivity. Most well-known operations however are not edge-breaking, which implies that they always preserve 3-connectivity. More specifically:
	
	\begin{corollary*}
		If $G$ is a simple, 3-connected embedded graph, then $\textit{Ambo}(G)$, $\textit{Truncation}(G)$,\\ $\textit{Capra}(G)$, $\textit{Chamfer}(G)$, $\textit{Snub}(G)$, $\textit{Gyro}(G)$, $\textit{Leapfrog}(G)$ and many more are 3-connected.
	\end{corollary*}

	It will become apparent that even within the class of edge-breaking operations Dual is special. It destroys 3-connectivity more easily than other lopsp-operations: If a simple 3-connected embedded graph has a simple dual, then Dual is the only lopsp-operation that may destroy its 3-connectivity.
	The foundations of these results were laid in the master's thesis \cite{vdc2020effect}. In that thesis only the less general lsp-operations were considered.
	
	This article is structured as follows: In Section~\ref{sec:definitions} we will give the definitions that are necessary for our results. It starts with the combinatorial definition of an embedded graph and after that lopsp-operations are defined. Section~\ref{sec:prep} is quite technical. There lemmas are proved that will be used in
	Section~\ref{sec:main} to prove our main theorems. In that section we prove that edge-preserving operations are the only lopsp-operations that preserve 3-connectivity in any embedded graph.

		\section{Definitions}\label{sec:definitions}
		
		\subsection*{Combinatorial definition of embedded graphs}
		
		A graph $G$ consists of a set $V_G$ of vertices and a set $E_G$ of edges. Each edge is \emph{incident} to exactly two vertices, which do not have to be different. An edge that is incident to the same vertex twice is called a \emph{loop}. If more than one edge is incident to the same two vertices these edges are called \emph{multiple edges}. An edge that is incident to vertices $x$ and $y$ is often denoted as $\{x,y\}$. Note that if there are multiple edges, this notation does not represent a unique edge. A graph that does not have loops or multiple edges is called \emph{simple}. To emphasize that a graph may not be simple we sometimes call it a \emph{multigraph}. 
		
		A \emph{walk} is an alternating sequence of vertices and edges of a graph such that each consecutive vertex and edge are incident. The first and last element of the sequence are vertices. If there are no multiple edges or if there is no risk of confusion a walk may be described by its sequence of vertices. A walk is a \emph{path} if it contains each vertex and each edge at most once. A walk is \emph{closed} if its first and last vertex are the same. A \emph{cycle} is a closed walk where only the first and last vertex are the same and all of its other vertices are different. If $v$ and $v'$ are vertices in a walk $P$ we write $P_{v,v'}$ for a segment of $P$ from $v$ to $v'$.
		
		In this text we will almost exclusively consider \emph{embedded graphs}. These are often defined from a topological perspective, but here we will use the purely combinatorial description using rotation systems. In a more topological context embedded graphs are often called \emph{maps}. With every edge we associate two \emph{oriented edges}. One for each incident vertex of the edge. This way every oriented edge has an associated vertex, which is incident to the corresponding edge of the graph. The inverse $e^{-1}$ of an oriented edge $e$ is the other oriented edge that corresponds to the same edge of the graph. A \emph{rotation system} is a permutation $\sigma$ of the set of oriented edges such that for any vertex $v$ of the graph, $\sigma$ induces a cyclic ordering on the oriented edges that have $v$ as their associated vertex. An embedded graph is a connected graph that has at least one edge, together with a rotation system. We will not specify this rotation system in our notation of embedded graphs, as we will only consider one rotation system for each graph. When drawing embedded graphs, the clockwise order of the edges around a vertex is the same as the cyclic order of the edges incident to that vertex induced by the rotation system.
		
		The oriented edges $e$ and $\sigma(e^{-1})$ form an \emph{angle}. A \emph{facial walk} or \emph{face} of an embedded graph is a cyclic sequence of oriented edges such that every two consecutive edges form an angle. Each oriented edge is in exactly one face. Let $F_G$ be the set of faces of the embedded graph $G$. The \emph{genus} of $G$ is given by the formula $\frac{2-|V_G| + |E_G| -|F_G|}{2}$. A graph is \emph{plane} if it has genus 0, i.e.\ if $|V_G| - |E_G| +|F_G|=2$.
		
		A graph $G'$ is a subgraph of a graph $G$ if $V_{G'}\subseteq V_G$ and $E_{G'}\subseteq E_G$. If all the edges of $E_G$ that have both their vertices in $V_{G'}$ are in $E_{G'}$ then $G'$ is an \emph{induced subgraph}. If $G$ is an embedded graph, then $G'$ with the embedding induced by the embedding of $G$ is an \emph{embedded subgraph} of $G$.
		
		We define a \emph{bridge} for a subgraph $G'$ of $G$. There are two kinds of bridges:
		\begin{itemize}
			\item If there is an edge in $G$ that is not in $G'$ but both its incident vertices are, then the embedded subgraph consisting of just this edge and its two vertices is a bridge.
			\item Let $C$ be a component of the subgraph of $G$ induced by the vertices of $G$ that are not in $G'$. The result of adding the edges of $G$ that have one vertex in $G'$ and one vertex in $C$ to $C$, together with their vertices in $G'$ is a bridge.
		\end{itemize}
		
		If edges $e$ and $e'$ form an angle in $G'$ and $e_b$ is an edge in a bridge such that $e$, $e_b$ and $e'$ appear in that order in the cyclic order around a vertex of $G'$, then the bridge is \emph{in} the angle $e$, $e'$ and the face containing it. All the vertices and edges in the bridge are also \emph{in} that face. To distinguish between vertices and edges in a bridge and vertices in the actual face, we say that the vertices and edges that are in the closed walk corresponding to the cyclic sequence of oriented edges are in the \emph{boundary} of the face. Vertices and edges that are in the face but not in the boundary are in the \emph{interior} of the face. If a bridge is in more than one face then those faces are \emph{bridged}. If a face is not bridged then it is \emph{simple}.
		
		If $f$ is a simple face in an embedded subgraph $G'$ of an embedded graph $G$, we will define the \emph{internal component} $IC(f)$. Informally, the internal component of $f$ is the embedded graph consisting of vertices and edges in $f$ that can be obtained by cutting along the facial walk. An example is shown in Figure~\ref{fig:IC}. More formally: Let $C$ be a cycle that has a vertex or edge for each vertex or edge in the walk $f$ respectively. If a vertex or edge appears more than once in the walk then for every appearance there is a different vertex or edge in $C$. The cycle $C$ has two faces. If we regard all the different copies of a vertex that appears more than once in $f$ as that one vertex of $f$, then for one face $f_C$ of $C$, the order of the vertices is the same as that of $f$. The other face of $C$ has those vertices in the reverse order. For every angle of $f$ there is exactly one angle in $f_C$ that corresponds to it. Now glue a copy of each bridge of $f$ into $f_C$ such that it is in the angles of $f_C$ corresponding to the angles of $f$ that the original bridge was in.	
		The result of this is $IC(f)$. 

		\begin{figure}
			\centering
			\scalebox{0.8}{\begin{tikzpicture}[scale=0.7, rotate = 180]
\tikzset{every node/.style={shape=circle, draw=black, scale=0.35, fill=black}}
\tikzset{H/.style={ultra  thick, draw=black}}
\tikzset{red/.style={thick, draw=red}}
\tikzset{rednode/.style={fill=red, draw=red}}

\node (a1) at (0,1) {};
\node (a2) at ({sin(72)}, {cos(72)}) {};
\node (a3) at ({sin(36)}, {-cos(36)}) {};
\node (a5) at ({-sin(72)}, {cos(72)}) {};
\node (a4) at ({-sin(36)}, {-cos(36)}) {};

\node (b1) at (0,2) {};
\node (b2) at ({2*sin(72)}, {2*cos(72)}) {};
\node (b3) at ({2*sin(36)}, {-2*cos(36)}) {};
\node (b5) at ({-2*sin(72)}, {2*cos(72)}) {};
\node (b4) at ({-2*sin(36)}, {-2*cos(36)}) {};

\node (c1) at (0,-2.5) {};
\node (c2) at ({2.5*sin(72)}, {-2.5*cos(72)}) {};
\node (c3) at ({2.5*sin(36)}, {2.5*cos(36)}) {};
\node (c5) at ({-2.5*sin(72)}, {-2.5*cos(72)}) {};
\node (c4) at ({-2.5*sin(36)}, {2.5*cos(36)}) {};

\node (d1) at (0,-3.5) {};
\node (d2) at ({3.5*sin(72)}, {-3.5*cos(72)}) {};
\node (d3) at ({3.5*sin(36)}, {3.5*cos(36)}) {};
\node (d5) at ({-3.5*sin(72)}, {-3.5*cos(72)}) {};
\node (d4) at ({-3.5*sin(36)}, {3.5*cos(36)}) {};

\begin{scope}[very thin]
\draw[H] (a1) -- (a2) -- (a3) -- (a4) -- (a5) -- (a1);
\draw (d1) -- (d2) -- (d3);
\draw[H] (d3) -- (d4);
\draw (d4) -- (d5) -- (d1);
\draw[H] (c3)  -- (b2) -- (c2) -- (b3) -- (c1) -- (b4) -- (c5) -- (b5) -- (c4);
\draw  (c4) -- (b1) -- (c3);
\draw (a1) -- (b1)
	(a2) -- (b2)
	(a3) -- (b3)
	(a4) -- (b4)
	(a5) edge[H] (b5);
	\draw (c1) -- (d1)
	(c2) -- (d2)
	(c3) edge[H] (d3)
	(c4) edge[H] (d4)
	(c5) -- (d5);
\end{scope}

\end{tikzpicture}}
			\qquad \qquad
			\scalebox{0.8}{\begin{tikzpicture}[scale=0.7, rotate = 180]
\tikzset{every node/.style={shape=circle, draw=black, scale=0.35, fill=black}}
\tikzset{H/.style={ultra  thick, draw=black}}
\tikzset{red/.style={thick, draw=red}}
\tikzset{rednode/.style={fill=red, draw=red}}

\node (a1) at (0,1) {};
\node (a2) at ({sin(72)}, {cos(72)}) {};
\node (a3) at ({sin(36)}, {-cos(36)}) {};
\node (a5) at ({-sin(72)}, {cos(72)}) {};
\node (a5') at ({-0.85*sin(72)}, {-0.5*cos(72)}) {};
\node (a4) at ({-sin(36)}, {-cos(36)}) {};

\node (b1) at (0,2) {};
\node (b2) at ({2*sin(72)}, {2*cos(72)}) {};
\node (b3) at ({2*sin(36)}, {-2*cos(36)}) {};
\node (b5) at ({-2*sin(72)}, {2*cos(72)}) {};
\node (b5') at ({-2.2*sin(72)}, {0.3*cos(72)}) {};
\node (b4) at ({-2*sin(36)}, {-2*cos(36)}) {};

\node (c1) at (0,-2.5) {};
\node (c2) at ({2.5*sin(72)}, {-2.5*cos(72)}) {};
\node (c3) at ({2.5*sin(36)}, {2.5*cos(36)}) {};
\node (c5) at ({-2.5*sin(72)}, {-2.5*cos(72)}) {};
\node (c4) at ({-2.5*sin(36)}, {2.5*cos(36)}) {};

\node[fill = none, draw = none] (d1) at (0,-3.5) {};
\node (d3) at ({3.5*sin(36)}, {3.5*cos(36)}) {};

\node (d4) at ({-3.5*sin(36)}, {3.5*cos(36)}) {};

\begin{scope}[very thin]
\draw[H] (a1) -- (a2) -- (a3) -- (a4) --(a5') (a5) -- (a1);

\draw[H] (d3) -- (d4);

\draw[H] (c3)  -- (b2) -- (c2) -- (b3) -- (c1) -- (b4) -- (c5) -- (b5') (b5)-- (c4);
\draw  (c4) -- (b1) -- (c3);
\draw (a1) -- (b1)
	(a2) -- (b2)
	(a3) -- (b3)
	(a4) -- (b4)
	(a5') edge[H] (b5')
	(a5) edge[H] (b5);
	\draw
	(c3) edge[H] (d3)
	(c4) edge[H] (d4);
\end{scope}

\end{tikzpicture}}
			\caption{On the left a graph is shown with a subgraph that is drawn with thicker edges. On the right the internal component of one of the faces of the subgraph is shown.}
			\label{fig:IC}
		\end{figure}
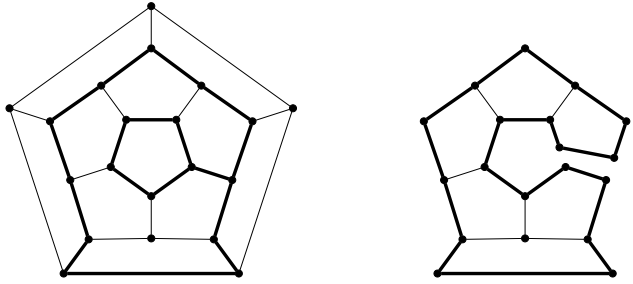

Next we define the barycentric subdivision of an embedded graph. It captures the topological structure of the embedding in a combinatorial way. Barycentric subdivisions and double chambers are used frequently throughout this text. 

	\begin{definition}	\label{def:bary}		
		The \emph{barycentric subdivision} $B_G$ of an embedded graph $G$ is an embedded graph with vertex set $V_G\cup E_G\cup F_G$. The vertices in $V_G$ are said to be of type 0, the vertices in $E_G$ are of type 1 and the vertices in $F_G$ are of type 2. Vertices of types 0 and 1 are adjacent in $B_G$ if the corresponding vertex and edge of $G$ are incident. A vertex of type 2 is adjacent to a vertex of type 0 or 1 if the boundary of the corresponding face of $G$ contains the corresponding vertex or edge. Vertices of the same type are not adjacent, which allows to define the \emph{type} of an edge as follows: An edge is of type $i\in\{0,1,2\}$ if its vertices are not of type $i$. To simplify notation we will often use the same names for the vertices, edges and faces of $G$ and their corresponding vertices of $B_G$. The graph $B_G$ is embedded in the obvious way such that all faces of $B_G$ are triangles and $gen(B_G)=gen(G)$. A part of a barycentric subdivision is shown in Figure~\ref{fig:barycentric}. In drawings, vertices and edges of types 0, 1 and 2 are colored red, green and black respectively. Edges of type 1 are dashed and edges of type 2 are dotted. 
	
		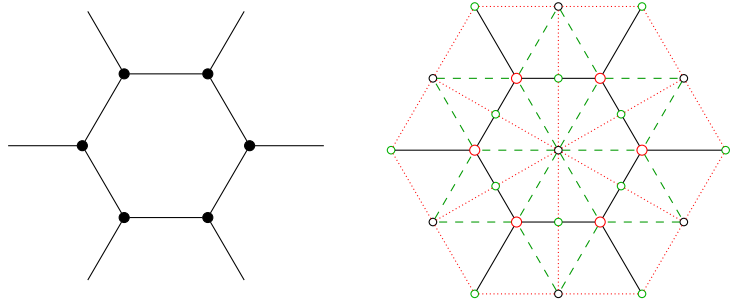
\begin{figure}
			\centering
			\begin{tikzpicture}[scale=1.1]
\tikzset{normal/.style={shape=circle, draw=black, scale=0.4, fill=black}}
\tikzset{type2/.style={shape=circle, draw=red, scale=0.5}}
\tikzset{noNode/.style={draw=none}}
\tikzset{0edge/.style={draw=red}}
\node[normal] (v1) at (-1,0) {};
\node[normal] (v4) at (1,0) {};
\node[normal] (v6) at ({cos(120)}, {-sin(120)}) {};
\node[normal] (v3) at ({cos(60)}, {sin(60)}) {};
\node[normal] (v5) at ({cos(60)}, {-sin(60)}) {};
\node[normal] (v2) at ({cos(120)}, {sin(120)}) {};

\node [noNode] (v7) at (-2, 0) {};
\node [noNode] (v8) at ({2*cos(120)}, {2*sin(120)}) {};
\node [noNode] (v9) at ({2*cos(60)}, {2*sin(60)}) {};
\node [noNode] (v10) at (2, 0) {};
\node [noNode] (v11) at ({2*cos(60)}, {-2*sin(60)}) {};
\node [noNode] (v12) at ({2*cos(120)}, {-2*sin(120)}) {};
\begin{scope}[]
\draw (v1) edge (v7);
\draw (v2) edge (v8);
\draw (v3) edge (v9);
\draw (v4) edge (v10);
\draw (v5) edge (v11);
\draw (v6) edge (v12);
\draw (v1) -- (v2) -- (v3) -- (v4) -- (v5) -- (v6) -- (v1);
\end{scope}
\end{tikzpicture}
			\quad
			\begin{tikzpicture}[scale=1.1]
\tikzset{normal/.style={shape=circle, draw=red, scale=0.4}}
\tikzset{type1/.style={shape=circle, draw=black!30!green, scale=0.3}}
\tikzset{type2/.style={shape=circle, draw=black, scale=0.3}}
\tikzset{0edge/.style={draw=red, densely dotted}}
\tikzset{1edge/.style={draw=black!40!green, dashed}}
\tikzset{2edge/.style={}}

\node[normal] (v1) at (-1,0) {};
\node[normal] (v4) at (1,0) {};
\node[normal] (v6) at ({cos(120)}, {-sin(120)}) {};
\node[normal] (v3) at ({cos(60)}, {sin(60)}) {};
\node[normal] (v5) at ({cos(60)}, {-sin(60)}) {};
\node[normal] (v2) at ({cos(120)}, {sin(120)}) {};

\node[type1] (e1) at ({(cos(120)-1)/2}, {sin(60)/2}) {};
\node[type2] (f1) at ({(cos(120)-1)}, {sin(60)}) {};
\node[type1] (e2) at (0, {sin(60)}) {};
\node[type2] (f2) at (0, {2*sin(60)}) {};
\node[type1] (e3) at({(1 - cos(120))/2}, {sin(60)/2}) {};
\node[type2] (f3) at ({(-cos(120)+1)}, {sin(60)}) {};
\node[type1] (e4) at({(1 - cos(120))/2}, {-sin(60)/2}) {};
\node[type2] (f4) at ({(-cos(120)+1)}, {-sin(60)}) {};
\node[type1] (e5) at({(0}, {-sin(60)}) {};
\node[type2] (f5) at (0, {-2*sin(60)}) {};
\node[type1] (e6) at ({(cos(120)-1)/2}, {-sin(60)/2}) {};
\node[type2] (f6) at ({(cos(120)-1)}, {-sin(60)}) {};

\node[type2] (f) at (0, 0) {};

\node [type1] (v7) at (-2, 0) {};
\node [type1] (v8) at ({2*cos(120)}, {2*sin(120)}) {};
\node [type1] (v9) at ({2*cos(60)}, {2*sin(60)}) {};
\node [type1] (v10) at (2, 0) {};
\node [type1] (v11) at ({2*cos(60)}, {-2*sin(60)}) {};
\node [type1] (v12) at ({2*cos(120)}, {-2*sin(120)}) {};

\begin{scope}[0edge]
\draw (f) -- (e1)--(f1);
\draw (f) -- (e2)--(f2);
\draw (f) -- (e3)--(f3);
\draw (f) -- (e4)--(f4);
\draw (f) -- (e5)--(f5);
\draw (f) -- (e6)--(f6);
\draw (f1)--(v8)--(f2)--(v9)--(f3)--(v10)--(f4)--(v11)--(f5)--(v12)--(f6)--(v7)--(f1);
\end{scope}

\begin{scope}[1edge]
\draw (f) edge (v1);
\draw (f) edge (v2);
\draw (f) edge (v3);
\draw (f) edge (v4);
\draw (f) edge (v5);
\draw (f) edge (v6);
\draw (v1)--(f1)--(v2)--(f2)--(v3)--(f3)--(v4)--(f4)--(v5)--(f5)--(v6)--(f6)--(v1);
\end{scope}

\begin{scope}[2edge]
\draw (v1) -- (e1) -- (v2) -- (e2) -- (v3) -- (e3) -- (v4) -- (e4) -- (v5) -- (e5) -- (v6) -- (e6) -- (v1);
\draw  (v1) edge (v7);
\draw  (v2) edge (v8);
\draw (v3) edge (v9);
\draw (v4) edge (v10);
\draw (v5) edge (v11);
\draw (v6) edge (v12);
\end{scope}

\end{tikzpicture}
			\caption{A face of an embedded graph $G$ and the corresponding part of $B_G$. }
			\label{fig:barycentric}
		\end{figure}

		Every face of $B_G$ is a triangle that has one vertex of each type. The faces of $B_G$ are called \emph{chambers}. In the literature these are also called \emph{flags}. The vertex of type $i$ in a chamber is called the \emph{$i$-vertex} of that chamber. The edge of type $i$ is the \emph{$i$-edge}.
		
		The subgraph of $B_G$ consisting of all its vertices and only the edges of types 1 and 2 is called $D_G$. The faces of this graph are called \emph{double chambers} because they can be considered the union of two chambers of $B_G$ that share their 0-edge. To avoid confusion with chambers, we call the 0-, 1- and 2-vertices of a double chamber the \emph{0-}, \emph{1-} and \emph{2-points}. The 1-edges are called \emph{1-sides} and the union of the two 2-edges and their shared vertex is called the \emph{2-side} of the double chamber.
	\end{definition}

\subsection*{lopsp-operations}
In the supplemental material of \cite{brinkmann2017comparing}, lopsp-operations are defined as patches that are cut out of tilings. To apply a lopsp-operation a copy of the patch is then glued into every double chamber. Instead of this definition using tilings, we repeat the definition of lopsp-operations from \cite{brinkmann2021local} here. This definition is very general and allows operations that for example introduce vertices of degree 1 or 2. In this text we will mostly restrict ourselves to $c3$-operations (Definition~\ref{def:c3}). Such operations are associated with 3-connected tilings. In practice, all well-known and used lopsp-operations are $c3$-operations. 

\begin{definition}\label{def:lopsp}
	Let $O$ be a 2-connected plane embedded multigraph with vertex set $V$, together with a labelling function $t: V \rightarrow \{0,1,2\}$ and three special vertices marked as $v_0$, $v_1$, and $v_2$. We say that a vertex is of \emph{type} $i$ if $t(v)=i$. 
	This embedded graph $O$ is a \emph{local orientation-preserving symmetry-preserving operation}, lopsp-operation for short, if the following properties hold:
	\begin{enumerate}[(1)]
		\item Every face is a triangle.
		\item There are no edges between vertices of the same type.
		\item We have
		\[t(v_0),t(v_2)\neq 1\]
		\[t(v_1)=1 \Rightarrow deg(v_1)=2\]
		and for each vertex $v$ different from $v_0$, $v_1$, and $v_2$:
			\[t(v)=1 \Rightarrow deg(v)=4\]
	\end{enumerate}
	
	\noindent
	We say that an edge is of \emph{type $i$} if it is not incident to a vertex of type $i$. This is well-defined because of (2). Note that the edges incident
	with a vertex are of two different types, and as every face is a triangle, these types appear in alternating order in the cyclic order of edges around the vertex. This implies that every
	vertex has an even degree. The requirement that $O$ is 2-connected is mentioned in the beginning, but would in fact also follow from the other conditions.
	
	Every face has exactly one vertex and one edge of each type. We will also call these faces \emph{chambers}.
\end{definition}

\begin{figure}
	\centering
	
	\scalebox{0.8}{\begin{tikzpicture}[scale=2]
\tikzset{type0/.style={shape=circle, draw=red, scale=0.4, fill=white}}
\tikzset{type1/.style={shape=circle, draw=black!30!green, scale=0.4, fill=white}}
\tikzset{type2/.style={shape=circle, draw=black, scale=0.4, fill=white}}
\tikzset{t0/.style={draw=red, densely dotted}}
\tikzset{t1/.style={draw=black!40!green, dashed}}
\tikzset{t2/.style={}}

\node[type0, scale=1.3,  label={[label distance=-0.02cm]90:{${v_0}$}}] (v0) at (0.5, 0) {};
\node[type1, scale=1.3, label={[label distance=-0.06cm]-90:{$v_1$}}] (v1) at (0.25, -0.5) {};
\node[type0, scale=1.3, label={[label distance=-0.02cm]90:{$v_2$}}] (v2) at (0, 1) {};
\node[type1] (e) at (0, 0) {};
\node[type0] (v) at (-0.5, 0) {};
\node[type2] (f) at (1, 0) {};
\node[type1] (E) at (-1, 0) {};

\begin{scope}[t0]
\draw (f)  edge[bend left=90, looseness=1.5] (E) ;
\draw (f)  edge[bend right=85] (E) ;
\draw (f)  edge[bend left=50] (e) ;
\draw (e)  edge[bend left=50] (f) ;
\draw (f)  edge[bend left=30] (v1) ;
\end{scope}

\begin{scope}[t1]
\draw[ultra thick] (v0) -- (f);
\draw[ultra thick] (f)  edge[bend right=45] (v2) ;
\draw (v)  edge[bend left=70] (f) ;
\draw (v)  edge[bend right=80, looseness=1.5] (f) ;
\draw (v)  edge[bend right=55, looseness=1] (f) ;
\end{scope}

\begin{scope}[t2]
\draw[ultra thick] (v0) -- (e) -- (v) ;
\draw (v) -- (E);
\draw (v2)  edge[bend right=45] (E) ;
\draw[ultra thick] (v1)  edge[bend left=30] (v) ;
\end{scope}

\end{tikzpicture}} \qquad
	\scalebox{0.8}{\begin{tikzpicture}[scale=3.7]
\tikzset{type0/.style={shape=circle, draw=red, scale=0.4, fill=white}}
\tikzset{type1/.style={shape=circle, draw=black!30!green, scale=0.4, fill=white}}
\tikzset{type2/.style={shape=circle, draw=black, scale=0.4, fill=white}}
\tikzset{t0/.style={draw=red, densely dotted}}
\tikzset{t1/.style={draw=black!40!green, dashed}}
\tikzset{t2/.style={}}

\node[type0, scale=1.3,  label={[label distance=-0.02cm]90:{$v_2$}}] (v) at (0, 0) {};


\node[type1] (e14) at ({cos(50+4*60)*cos(15)/cos(20)}, {sin(50+4*60)*cos(15)/cos(20)}) {};
\node[type0] (v14) at ({cos(40+4*60)*cos(15)/cos(10)}, {sin(40+4*60)*cos(15)/cos(10)}) {};
\node[type1,  scale=1.3, label={[label distance=-0.02cm]-90:{$v_1$}}] (e24) at ({cos(30+4*60)*cos(15)}, {sin(30+4*60)*cos(15)}) {};
\node[type0] (v24) at ({cos(20+4*60)*cos(15)/cos(10)}, {sin(20+4*60)*cos(15)/cos(10)}) {};
\node[type1] (e34) at ({cos(10+4*60)*cos(15)/cos(20)}, {sin(10+4*60)*cos(15)/cos(20)}) {};
\node[type0,   scale=1.3,  label={[label distance=-0.02cm]120:{$v_{0,L}$}}] (v34) at ({cos(4*60)*(cos(15)/cos(30))}, {sin(4*60)*(cos(15)/cos(30))}) {};
\node[type1] (e44) at ({cos(40+4*60)*cos(15)/cos(10)/2}, {sin(40+4*60)*cos(15)/cos(10)/2}) {};
\node[type2] (f4) at ({cos(4*60)*(cos(15)/cos(30))/2}, {sin(4*60)*(cos(15)/cos(30))/2}) {};

\node[type2] (f5) at ({cos(5*60)*(cos(15)/cos(30))/2}, {sin(5*60)*(cos(15)/cos(30))/2}) {};
\node[type0,  scale=1.3,  label={[label distance=-0.02cm]60:{$v_{0,R}$}}] (v35) at ({cos(5*60)*(cos(15)/cos(30))}, {sin(300)*(cos(15)/cos(30))}) {};

\draw[t2,ultra thick] (v35) -- (e14) -- (v14) --(e24) ;
\draw[t2,ultra thick] (e24) -- (v24) -- (e34) -- (v34) ;
\draw[t2] (v14) -- (e44)--(v);
\draw[t1] (v14) -- (f4) -- (v24);
\draw[t1,ultra thick] (v) -- (f4) -- (v34);
\draw[t0] (e24) -- (f4) -- (e34);
\draw[t0] (f4) -- (e44);
\draw[t0] (e14) -- (f5) -- (e44);
\draw[t1] (v14) -- (f5);

\draw[t1,ultra thick] (v) -- (f5) -- (v35);

\node[draw=white] (k) at (0.5,-1.1) {};

\end{tikzpicture}}
	\caption{\label{fig:lopsp_gyro}On the left, the lopsp-operation gyro is shown. The thicker edges are the edges of a cut-path $P$. On the right the corresponding double chamber patch $O_P$ is drawn. The two copies of $P$ in $O_P$ are drawn thicker.}
\end{figure}
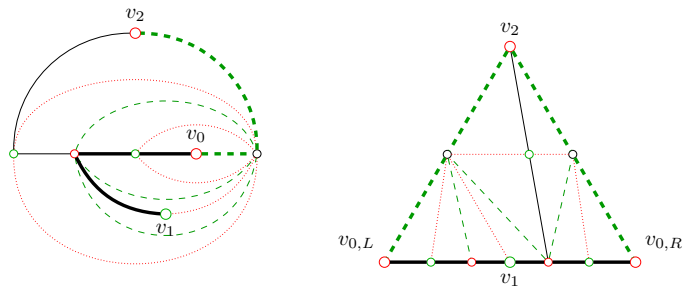

{\bf Application of a lopsp-operation:} Before applying a lopsp-operation $O$ a \emph{cut-path} $P$ needs to be chosen. This is a (simple) path in $O$ from $v_1$ to $v_2$ that contains $v_0$. As lopsp-operations are 2-connected, Menger's theorem implies that such a path exists for any lopsp-operation $O$. The unique face of $P$ as a subgraph of $O$ is simple. The internal component of this face is called the \emph{double chamber patch} $O_P$ of $O$ with respect to the cut-path $P$. This plane graph contains two copies of the path $P$. Together they form the \emph{outer face} of $O_P$. As endpoints, the vertices $v_1$ and $v_2$ appear in both copies of $P$ and therefore there is only one copy of each of them in $O_P$. However, there are two copies $v_{0,L}$ and $v_{0,R}$ of $v_0$ in $O_P$. Figure~\ref{fig:lopsp_gyro} shows the lopsp-operation gyro on the left and a double chamber patch of gyro on the right.

To apply $O$ to an embedded graph $G$ the edges of $D_G$ are replaced by copies of $P_{v_0,v_1}$ and $P_{v_0,v_2}$, and then a copy of $O_P$ is glued into every face of that graph. An edge of type $i$ is replaced by a copy of $P_{v_0,v_i}$ such that the vertex of type $j$ is identified with the copy of $v_j$. The result is an embedded graph $B_{O(G)}$ which is the barycentric subdivision of another embedded graph $O(G)$. This graph $O(G)$ is the result of applying $O$ to $G$. Note that this notation contains no reference to the chosen cut-path $P$. This is justified by results in \cite{brinkmann2021local} where it is proved that the result of a lopsp-operation is independent of the chosen cut-path.
In that article a map $\pi$ is also defined. We repeat that definition here, together with the definition of sides and points in $B_{O(G)}$.

\begin{definition}\label{def:pi}
	As $B_{O(G)}$ is obtained by gluing copies of $O_P$ together, every vertex, edge, or face of $B_{O(G)}$ is a copy of exactly one vertex, edge, or face of $O$ respectively. The surjective map $\pi$ maps every vertex, edge and face of $B_{O(G)}$ to its
	corresponding vertex, edge, or face of $O$. 
	
	The map $\pi^{-1}$ is defined for subgraphs of $O$. If $H$ is a subgraph of $O$, then $\pi^{-1}(H)$ is the subgraph of $B_{O(G)}$ consisting of all vertices and edges whose image under $\pi$ is in $H$. If $\pi^{-1}(H)$ is connected, then it is an embedded subgraph of $B_{O(G)}$ with the embedding induced by the embedding of $B_{O(G)}$.
	
	For $i\in \{0,1,2\}$ an \emph{$i$-point} in $B_{O(G)}$ is a vertex $v$ with $\pi(v)=v_i$. When applying a lopsp-operation with a cut-path $P$, a copy of $P_{v_0,v_2}$ in $B_{O(G)}$ between a 0-point and a 2-point is a \emph{1-side}. Two copies of $P_{v_0,v_1}$ in $B_{O(G)}$ with the same 1-point are a \emph{2-side}. Note that each of these points and sides corresponds to exactly one point or side of $D_G$.
\end{definition}

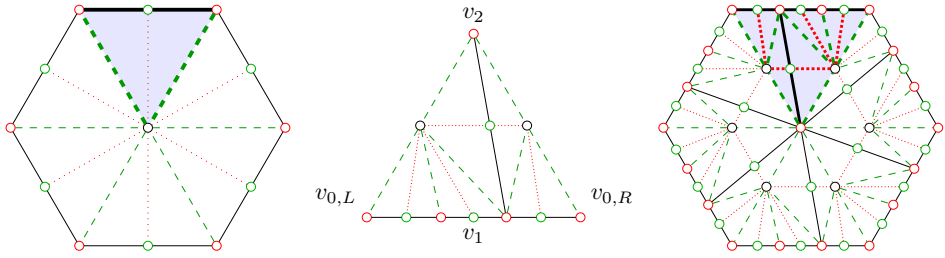
\begin{figure}
	\centering
	\scalebox{0.9}{\begin{tikzpicture}[scale=1.8]
\tikzset{type0/.style={shape=circle, draw=red, scale=0.4, fill=white}}
\tikzset{type1/.style={shape=circle, draw=black!30!green, scale=0.4, fill=white}}
\tikzset{type2/.style={shape=circle, draw=black, scale=0.4, fill=white}}
\tikzset{0edge/.style={draw=red, dotted}}
\tikzset{1edge/.style={draw=black!40!green, dashed}}
\tikzset{2edge/.style={}}

\fill[blue!10!white] (0,0) -- ({-cos(15)*tan(30)}, {cos(15)})-- (0, {cos(15)}) -- ({cos(15)*tan(30)}, {cos(15)}) -- cycle;

\node[type2] (v2) at (0, 0) {};

\node[type1] (v1) at (0, {cos(15)}) {};
\node[type1] (v11) at ({cos(15) *cos(30)}, {cos(15)*sin(30)}) {};
\node[type1] (v12) at ({cos(15) *cos(30)}, {-cos(15)*sin(30)}) {};
\node[type1] (v13) at (0, {-cos(15)}) {};
\node[type1] (v14) at ({-cos(15) *cos(30)}, {-cos(15)*sin(30)}) {};
\node[type1] (v15) at ({-cos(15) *cos(30)}, {cos(15)*sin(30)}) {};

\node[type0] (v0) at ({cos(15)*tan(30)}, {cos(15)}) {};
\node[type0] (v01) at ({cos(15)/sin(60)}, 0) {};
\node[type0] (v02) at ({cos(15)*tan(30)}, {-cos(15)}) {};
\node[type0] (v03) at ({-cos(15)*tan(30)}, {-cos(15)}) {};
\node[type0] (v04) at ({-cos(15)/sin(60)}, 0) {};
\node[type0] (v05) at ({-cos(15)*tan(30)}, {cos(15)}) {};

\begin{scope}[0edge]
\draw[] (v1) -- (v2) ;
\draw (v2) -- (v11);
\draw (v12) -- (v2) -- (v13);
\draw (v14) -- (v2) -- (v15);
\end{scope}

\begin{scope}[2edge]
\draw[ultra thick] (v05)--(v1) -- (v0);
\draw (v0) --(v11)--(v01) -- (v12) -- (v02) -- (v13) 
-- (v03) -- (v14) -- (v04) -- (v15) -- (v05);
\end{scope}

\begin{scope}[1edge]
\draw[ultra thick] (v0) -- (v2) --(v05);
\draw (v2) -- (v01);
\draw (v02) -- (v2) -- (v03);
\draw (v04) -- (v2);
\end{scope}

\end{tikzpicture}} 
	\scalebox{0.9}{\begin{tikzpicture}[scale=2.8]
\tikzset{type0/.style={shape=circle, draw=red, scale=0.4, fill=white}}
\tikzset{type1/.style={shape=circle, draw=black!30!green, scale=0.4, fill=white}}
\tikzset{type2/.style={shape=circle, draw=black, scale=0.4, fill=white}}
\tikzset{t0/.style={draw=red, densely dotted}}
\tikzset{t1/.style={draw=black!40!green, dashed}}
\tikzset{t2/.style={}}

\node[type0,   label={[label distance=-0.02cm]90:{$v_2$}}] (v) at (0, 0) {};


\node[type1] (e14) at ({cos(50+4*60)*cos(15)/cos(20)}, {sin(50+4*60)*cos(15)/cos(20)}) {};
\node[type0] (v14) at ({cos(40+4*60)*cos(15)/cos(10)}, {sin(40+4*60)*cos(15)/cos(10)}) {};
\node[type1, label={[label distance=-0.02cm]-90:{$v_1$}}] (e24) at ({cos(30+4*60)*cos(15)}, {sin(30+4*60)*cos(15)}) {};
\node[type0] (v24) at ({cos(20+4*60)*cos(15)/cos(10)}, {sin(20+4*60)*cos(15)/cos(10)}) {};
\node[type1] (e34) at ({cos(10+4*60)*cos(15)/cos(20)}, {sin(10+4*60)*cos(15)/cos(20)}) {};
\node[type0,  label={[label distance=-0.02cm]120:{$v_{0,L}$}}] (v34) at ({cos(4*60)*(cos(15)/cos(30))}, {sin(4*60)*(cos(15)/cos(30))}) {};
\node[type1] (e44) at ({cos(40+4*60)*cos(15)/cos(10)/2}, {sin(40+4*60)*cos(15)/cos(10)/2}) {};
\node[type2] (f4) at ({cos(4*60)*(cos(15)/cos(30))/2}, {sin(4*60)*(cos(15)/cos(30))/2}) {};

\node[type2] (f5) at ({cos(5*60)*(cos(15)/cos(30))/2}, {sin(5*60)*(cos(15)/cos(30))/2}) {};
\node[type0,  label={[label distance=-0.02cm]60:{$v_{0,R}$}}] (v35) at ({cos(5*60)*(cos(15)/cos(30))}, {sin(300)*(cos(15)/cos(30))}) {};

\draw[t2] (v35) -- (e14) -- (v14) --(e24) ;
\draw[t2] (e24) -- (v24) -- (e34) -- (v34) ;
\draw[t2] (v14) -- (e44)--(v);
\draw[t1] (v14) -- (f4) -- (v24);
\draw[t1] (v) -- (f4) -- (v34);
\draw[t0] (e24) -- (f4) -- (e34);
\draw[t0] (f4) -- (e44);
\draw[t0] (e14) -- (f5) -- (e44);
\draw[t1] (v14) -- (f5);

\draw[t1] (v) -- (f5) -- (v35);

\end{tikzpicture}} 
	\scalebox{0.9}{\begin{tikzpicture}[scale=1.8]
\tikzset{type0/.style={shape=circle, draw=red, scale=0.4, fill=white}}
\tikzset{type1/.style={shape=circle, draw=black!30!green, scale=0.4, fill=white}}
\tikzset{type2/.style={shape=circle, draw=black, scale=0.4, fill=white}}
\tikzset{t0/.style={draw=red, densely dotted}}
\tikzset{t1/.style={draw=black!40!green, dashed}}
\tikzset{t2/.style={}}

\fill[blue!10!white] (0,0) -- ({-cos(15)*tan(30)}, {cos(15)}) -- (0, {cos(15)}) -- ({cos(15)*tan(30)}, {cos(15)}) -- cycle;

\node[type0] (v) at (0, 0) {};

\foreach \i in {0,...,5}
{
\node[type1] (e1\i) at ({cos(50+\i*60)*cos(15)/cos(20)}, {sin(50+\i*60)*cos(15)/cos(20)}) {};
\node[type0] (v1\i) at ({cos(40+\i*60)*cos(15)/cos(10)}, {sin(40+\i*60)*cos(15)/cos(10)}) {};
\node[type1] (e2\i) at ({cos(30+\i*60)*cos(15)}, {sin(30+\i*60)*cos(15)}) {};
\node[type0] (v2\i) at ({cos(20+\i*60)*cos(15)/cos(10)}, {sin(20+\i*60)*cos(15)/cos(10)}) {};
\node[type1] (e3\i) at ({cos(10+\i*60)*cos(15)/cos(20)}, {sin(10+\i*60)*cos(15)/cos(20)}) {};
\node[type0] (v3\i) at ({cos(\i*60)*(cos(15)/cos(30))}, {sin(\i*60)*(cos(15)/cos(30))}) {};
\node[type1] (e4\i) at ({cos(40+\i*60)*cos(15)/cos(10)/2}, {sin(40+\i*60)*cos(15)/cos(10)/2}) {};
\node[type2] (f\i) at ({cos(\i*60)*(cos(15)/cos(30))/2}, {sin(\i*60)*(cos(15)/cos(30))/2}) {};
}

\foreach \j / \i in {1/0,4/3,5/4,0/5}
{
\draw[t2] (v3\j) -- (e1\i) -- (v1\i) --(e2\i) -- (v2\i) -- (e3\i) -- (v3\i) ;
\draw[t2] (v1\i) -- (e4\i)--(v);
\draw[t1] (v1\i) -- (f\i) -- (v2\i);
\draw[t1] (v) -- (f\i) -- (v3\i);
\draw[t0] (e2\i) -- (f\i) -- (e3\i);
\draw[t0] (f\i) -- (e4\i);
\draw[t0] (e1\i) -- (f\j) -- (e4\i);
\draw[t1] (v1\i) -- (f\j);
}

\foreach \j / \i in {2/1}
{
\begin{scope}[very thick]
\draw[t2,very thick] (v3\j) -- (e1\i) -- (v1\i) --(e2\i) -- (v2\i) -- (e3\i) -- (v3\i) ;
\draw[t2,very thick] (v1\i) -- (e4\i)--(v);
\draw[t1,very thick] (v1\i) -- (f\i) -- (v2\i);
\draw[t1,very thick] (v) -- (f\i) -- (v3\i);
\draw[t0,very thick] (e2\i) -- (f\i) -- (e3\i);
\draw[t0,very thick] (f\i) -- (e4\i);
\draw[t0,very thick] (e1\i) -- (f\j) -- (e4\i);
\draw[t1,very thick] (v1\i) -- (f\j);
\end{scope}
}

\foreach \j / \i in {3/2}
{
\draw[t2] (v3\j) -- (e1\i) -- (v1\i) --(e2\i) -- (v2\i) -- (e3\i) -- (v3\i) ;
\draw[t2] (v1\i) -- (e4\i)--(v);
\draw[t1] (v1\i) -- (f\i) -- (v2\i);
\draw[t1, very thick] (v) -- (f2) -- (v32);
\draw[t0] (e2\i) -- (f\i) -- (e3\i);
\draw[t0] (f\i) -- (e4\i);
\draw[t0] (e1\i) -- (f\j) -- (e4\i);
\draw[t1] (v1\i) -- (f\j);
}

\end{tikzpicture}}
	
	\caption{\label{fig:gyro_applied}On the left, the barycentric subdivision of a hexagonal face is shown. On the right, the lopsp-operation gyro -- depicted in the middle -- is applied to it. The blue shaded area shows one double chamber.}
\end{figure}

The most prominent lopsp-operation is without doubt Dual. As it will play a special role in this article we describe it in more detail. For an embedded graph $G$, the dual $G^*$ of $G$ is the graph that has the set of faces of $G$ as its set of vertices, and two vertices are adjacent if and only if the corresponding faces of $G$ share an edge. The vertices of $G$ become the faces in $G^*$. In short, Dual switches the roles of faces and vertices. The lopsp-operation Dual is shown in Figure~\ref{fig:dual}. It has only three vertices, three edges and two faces. The vertices $v_0$, $v_1$ and $v_2$ are of types 2, 1 and 0 respectively. It follows that if we ignore the labelling, $B_{G^*}$ and $B_G$ are the same embedded graph. The only difference is that the types of the vertices of types 0 and 2 are switched.

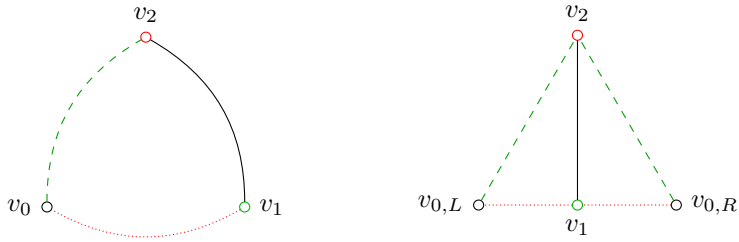
\begin{figure}
	\centering
	\begin{tikzpicture}[scale=1.3, baseline = -14pt]
\tikzset{type0/.style={shape=circle, draw=red, scale=0.4, fill=white}}
\tikzset{type1/.style={shape=circle, draw=black!30!green, scale=0.4, fill=white}}
\tikzset{type2/.style={shape=circle, draw=black, scale=0.4, fill=white}}
\tikzset{t0/.style={draw=red, densely dotted}}
\tikzset{t1/.style={draw=black!40!green, dashed}}
\tikzset{t2/.style={}}

\node[type2, label={[label distance=-0.0cm]180:$v_0$}] (v0) at (-1, 0) {};
\node[type1, label={[label distance=-0.0cm]0:$v_1$}] (v1) at (1,0) {};
\node[type0, label={[label distance=-0.0cm]90:$v_2$}] (v2) at (0, {sqrt(3)}) {};

\begin{scope}[t0]
\draw (v1) edge[bend left] (v0);
\end{scope}

\begin{scope}[t1]
\draw (v0) edge[bend left] (v2);
\end{scope}

\begin{scope}[t2]
\draw (v1) edge[bend right] (v2);
\end{scope}


\end{tikzpicture}\qquad\qquad
	\begin{tikzpicture}[scale=1.3]
\tikzset{type0/.style={shape=circle, draw=red, scale=0.4, fill=white}}
\tikzset{type1/.style={shape=circle, draw=black!30!green, scale=0.4, fill=white}}
\tikzset{type2/.style={shape=circle, draw=black, scale=0.4, fill=white}}
\tikzset{t0/.style={draw=red, densely dotted}}
\tikzset{t1/.style={draw=black!40!green, dashed}}
\tikzset{t2/.style={}}

\node[type2, label={[label distance=-0.0cm]180:$v_{0,L}$}] (v0L) at (-1, 0) {};
\node[type2, label={[label distance=-0.0cm]0:$v_{0,R}$}] (v0R) at (1,0) {};
\node[type0, label={[label distance=-0.0cm]90:$v_2$}] (v2) at (0, {sqrt(3)}) {};
\node[type1, label={[label distance=-0.0cm]-90:$v_1$}] (v1) at (0, 0) {};

\begin{scope}[t0]
\draw (v1) -- (v0L);
\draw (v0R) -- (v1);
\end{scope}

\begin{scope}[t1]
\draw (v0L) -- (v2);
\draw (v0R) -- (v2);
\end{scope}

\begin{scope}[t2]
\draw (v1) -- (v2);
\end{scope}

\end{tikzpicture}
	\caption{The left image shows the lopsp-operation Dual, the right one shows the unique double chamber patch for this operation.}\label{fig:dual}
\end{figure}

A 4-cycle in a barycentric subdivision, double chamber patch or multiple double chamber patches glued together is called {\em trivial} if it has a face with only one edge or only a single type-1 vertex and its four incident edges in its interior.

\begin{definition}\label{def:c3}
	An embedded graph is \emph{3-connected} if it has at least four vertices and it has no vertex-cut of fewer than three vertices. 
 
	An embedded graph is a \emph{polyhedral embedding} if the intersection of any two of its faces is either empty, one vertex or one edge. Every polyhedral embedding is 3-connected.
	
	A lopsp-operation is \emph{$c3$} if for each polyhedral embedding $G$, the embedded graph $O(G)$ is also a polyhedral embedding. 
\end{definition}

The definition of a $c3$-lopsp-operation is different from the initial one in \cite{brinkmann2021local}, but Theorem 19 in \cite{brinkmann2021local} states that it is an equivalent definition. Lemma~\ref{lem:c3char} will be useful to prove that a lopsp-operation is not $c3$.

\begin{lemma}\label{lem:c3char}
	If there exists a cut-path $P$ for a lopsp-operation $O$ such that there is a 2-cycle or a non-trivial 4-cycle in two copies of $O_P$ sharing exactly one side and not their other point, then $O$ is not $c3$.
\end{lemma}
\begin{proof}
	For any polyhedral embedding $G$ we can find a subgraph $H$ of $B_{O(G)}$ that is isomorphic to these two copies of $O_P$. It follows that there is also a 2-cycle or non-trivial 4-cycle in $B_{O(G)}$. Lemma 14 in \cite{brinkmann2021local} now implies that $O(G)$ is not polyhedral, a contradiction with the definition of $c3$-operations.	
\end{proof}

We use Lemma~\ref{lem:size_OG} to make sure that the result of applying a lopsp-operation is not too small. If a graph is 3-connected then the result of applying a lopsp-operation different from Dual to it also has at least four vertices.

\begin{lemma}\label{lem:size_OG}
	Let $G$ be an embedded graph that has at least four vertices and $|V_G|$ edges (it is not a tree), and let $O$ be a $c3$-lopsp-operation different from Dual. Then the number of vertices of $O(G)$ is at least the number of vertices of $G$.
\end{lemma}
\begin{proof}
	Assume that $|V_{O(G)}|<|V_G|$. For every edge of $G$ there are exactly two double chambers in $D_G$, so there are $2|E_G|$ double chambers in $D_G$. For any vertex $x$ in $O$ that is not $v_0$, $v_1$ or $v_2$, the number of vertices in $\pi^{-1}(x)$ is exactly the number of double chambers. This implies that no vertex of $O$ except $v_0$, $v_1$, and $v_2$ is of type 0, otherwise there would be at least $2|E_G|\geq 2|V_G|> |V_G|$ vertices in $O(G)$. If $v_0$ would be of type 0 then $O(G)$ would have at least $|V_G|$ vertices. If $v_1$ would be of type 0 then $O(G)$ would have at least $|E_G|\geq|V_G|$ vertices. This implies that only $v_2$ is of type 0. With Lemma~\ref{lem:c3char} it follows that $O$ is Dual.
\end{proof}

\section{Preparatory results}\label{sec:prep}
	Since c3-lopsp-operations are associated with 3-connected tilings, one might expect that $c3$-operations preserve 3-connectivity. However, it is known that Dual does not always preserve 3-connectivity \cite{bokal2022connectivity}, so it is not true that all lopsp- or even lsp-operations do. In Section~\ref{sec:main} it is proved for a large subclass of lopsp-operations including most well-known operations such as Truncation and Ambo that they always preserve 3-connectivity in simple embedded graphs. We also prove that all operations not belonging to this subclass, for example Join and Dual, destroy 3-connectivity in certain embedded graphs. In this section we will prove a number of technical results which will allow us to prove the results in Section~\ref{sec:main}.
	 
	An informal explanation of our approach is the following:
	Every vertex of $G$ is either also a vertex of $O(G)$, or it becomes a face in $O(G)$. We will call this vertex or set of vertices in the face the vertex-shadow (Definition~\ref{def:0neighbourhood}). For every edge of $G$ at least one path between its vertex-shadows can be found that is completely contained in the two double chambers with that edge as their 1-vertex (Lemma~\ref{lem:edge-path}). The vertex-shadows together with these paths represent the structure of $G$ in $O(G)$. There are two ways how $G$ can be 3-connected while $O(G)$ is not. One way is that there are `local cuts', which cut off vertices of $O(G)$ from the underlying structure of $G$ in $O(G)$, while this structure stays intact. It will follow from Lemma~\ref{lem:no_2_in_v0set} that these cuts cannot exist under our conditions. The other way that 3-connectivity may not be preserved by $O$ is if the structure of $G$ in $O(G)$ is `broken'. This can happen at the level of vertices (Definition~\ref{def:break_vertex}) or edges (Definition~\ref{def:edge_breaking}). 
	
	In the following we will often consider two double chambers sharing their 1-vertex. Also in figures we often draw two double chamber patches instead of just one.
	
\begin{definition}\label{def:diamond}
	For an embedded graph $G$, the union of the two double chambers of $D_G$ with 1-vertex $e$ is the \emph{diamond} around $e$. We also say that the \emph{$P$-diamond} of a lopsp-operation $O$ with cut-path $P$ is the embedded graph consisting of two copies of $O_P$ that share their copies of $P_{v_0,v_1}$. If the cut-path is not specified we refer to such a structure as a \emph{diamond} of $O$.
\end{definition}

	The following definition is used to describe what happens to vertices, edges, and faces when a lopsp-operation is applied. The definition will be used in barycentric subdivisions, lopsp-operations and double chamber patches.

	\begin{definition}\label{def:0neighbourhood}	
		Let $H$ be an embedded graph with labelling function $t:V_H\to \{0,1,2\}$. The \emph{0-neighbourhood} $N_0(v)$ of a vertex $v$ in $H$ is the set of all vertices in $H$ that are of type 0 and are adjacent or equal to $v$, i.e.
		\[N_0(v)=\left\{ x\in V_{H} \mid \text{$t(x)=0$, $x=v$ or $\{x,v\}\in E_H$} \right\}.\]
		
		Let $v$ be a vertex of an embedded graph $G$, and let $v'$ be the corresponding vertex of $B_{O(G)}$. The 0-neighbourhood $N_0(v')$ consists of vertices of type 0 in $B_{O(G)}$. The set of vertices of $O(G)$ corresponding to those in $N_0(v')$ is called the \emph{vertex-shadow} $S_O(v)$ of $v$. Similarly, for a face $f$ of $G$, we say the set of vertices of $O(G)$ corresponding to the vertices in the 0-neighbourhood of $f$ in $B_{O(G)}$ is the \emph{face-shadow} $S_O(f)$ of $f$. 
	\end{definition}

\begin{lemma}\label{lem:image_singleton_or_face}
	Let $G$ be an embedded graph and let $O$ be a lopsp-operation. A vertex-shadow or face-shadow of a vertex or face of $G$ is one of the following:
	\begin{itemize}
		\item A singleton, containing one vertex of $O(G)$.
		\item A set consisting of all the vertices of a face of $O(G)$.
	\end{itemize}
\end{lemma}
\begin{proof}
	This follows immediately from the fact that in a lopsp-operation $v_0$ and $v_2$ are always of type 0 or type 2. 
\end{proof}

With this definition we can easily characterise Dual and the identity operation.

\begin{lemma}\label{lem:dual_char}
	A $c3$ lopsp-operation is Dual if and only if $v_2\in N_0(v_0)$ and it is the identity operation if and only if $v_0\in N_0(v_2)$. An operation is Dual or the identity operation if and only if there is an edge between $v_0$ and $v_2$.
\end{lemma}
\begin{proof}
	It is evident that Dual and the identity operation satisfy the conditions.
	Assume that there is an edge between $v_0$ and $v_2$. As neither $v_0$ nor $v_2$ can be of type 1, the edge $\{v_0,v_2\}$ is of type 1. This implies that $v_2\in N_0(v_0)$ or $v_0\in N_0(v_2)$. Adding this edge to any path from $v_0$ to $v_1$ without $v_2$ gives a cut-path $P$.
	
	The 1-edge between $v_0$ and $v_2$ induces a cycle of length four in the $P$-diamond of $O$, which is trivial by Lemma~\ref{lem:c3char}. As there is at least one vertex --namely $v_1$-- on the inside of this cycle, it is the only one and it is of type 1. This implies that $O$ is Dual or the identity operation, depending on the types of $v_0$ and $v_2$.
\end{proof}

Lemma~\ref{lem:intersecting_edges} and Lemma~\ref{lem:minimal_path} are useful for many proofs in this section.

\begin{lemma}\label{lem:intersecting_edges}
	Let $O$ be a lopsp-operation with a cut-path $P$. Let $x$ and $y$ be two vertices of $P$ different from $v_1$ and $v_2$. Let $P_L$ be one copy of $P$ in $O_P$ and let $P_R$ be the other. Let $x_L$ and $y_L$ be the copies of $x$ and $y$ in $P_L$ and let $x_{R}$ and $y_{R}$ be the copies of $x$ and $y$ in $P_R$. If there is an edge in $O_P$ between $x_L$ and $y_R$, then any path in $O_P$ between $y_L$ and $x_R$ also contains $x_L$ or $y_R$. In particular, there is no edge between $x_R$ and $y_L$.
\end{lemma}
\begin{proof}
	As $x$ and $y$ are not $v_1$ or $v_2$, $x_L\neq x_{R}$ and $y_L\neq y_{R}$. Assume that there is an edge between $x_L$ and $y_R$ in $O_P$. Let $O_P^+$ be the plane graph which is the result of adding one extra edge between $x_L$ and $y_R$ to $O_P$ in the outer face. Now the two edges between $x_L$ and $y_R$ form a cycle in a plane graph, and $y_L$ and $x_R$ are on different sides of that cycle. The Jordan curve theorem implies that any path between $y_L$ and $x_R$ in $O_P^+$ must also contain $x_L$ or $y_R$, which proves the lemma.
\end{proof}

\begin{lemma}\label{lem:minimal_path}
	Let $O$ be a $c3$-lopsp-operation and let $G$ be an embedded graph. Let $O$ be applied to $G$ using a cut-path of minimal length. If the vertices of an edge of $B_{O(G)}$ are on the same 1- or 2-side of a double chamber, then the edge itself is on that 1- or 2-side.
\end{lemma}
\begin{proof}
	This follows immediately from Lemma~15 in \cite{brinkmann2021local} and Lemma~\ref{lem:c3char} in this article.
\end{proof}

Lemma~\ref{lem:4cycle_positions} will only be used in Section~\ref{sec:main}. It gives insight in where 4-cycles in $B_{O(G)}$ can be located with respect to the double chambers of $G$. 

\begin{lemma}\label{lem:4cycle_positions}
	Let $O$ be a $c3$-lopsp-operation that is not Dual or the identity operation and let $G$ be a simple 3-connected embedded graph such that there are no multiple edges in $B_G$. Then there are no cycles of length two in $B_{O(G)}$. If $c$ is a non-trivial cycle in $B_{O(G)}$ of length 4, then two non-adjacent vertices of $c$ are 2-points. The other 2 vertices are on different 2-sides. 
\end{lemma}
\begin{proof}
	It follows from Lemma~\ref{lem:dual_char} that there is no edge between $v_0$ and $v_2$.
	
	Let $c$ be a non-trivial cycle of length two or four in $B_{O(G)}$. We apply $O$ to $G$ using a cut-path of minimal length.
	Let $M$ be a set of double chambers in $D_G$ of minimal size such that every edge of $c$ is in at least one double chamber of $M$. If $M$ has more than one element, then as $c$ must enter and leave every double chamber in $M$, at least two vertices on the boundary of each double chamber in $M$ are in $c$. It follows that every double chamber of $M$ shares at least two points with other double chambers in $M$. Note that every double chamber in $M$ contains at least one edge of $c$ that is not in another double chamber of $M$, as otherwise $M$ would not be of minimal size. It follows that $M$ has at most four elements.

	Assume that there is a 2-point $f$ such that the union of all the double chambers of $M$ containing this 2-point contains only one edge of $c$. The vertices of this edge would be on the same 2-side of a double chamber $D$. At least one of the vertices of the edge is not a 0-point, so the double chamber $D'$ sharing the 2-side with $D$ also contains an edge of $c$ and is therefore in $M$. Lemma~\ref{lem:minimal_path} implies that the edge of $c$ in $D$ is on the shared 2-side, which implies that it is also in $D'$ and therefore $M$ is not minimal, a contradiction. It follows that there are at most two different 2-points in the double chambers of $M$.
	
	Note that if the configuration of the double chambers that appear in $M$ can also be found in a polyhedral embedding $H$, then $c$ induces a non-trivial cycle of length two or four in $B_{O(H)}$. That is not possible as $O(H)$ is 3-connected (Lemma 14 and Theorem 19 in \cite{brinkmann2021local}).
	There are two cases:
	
	\begin{itemize}
		\item \textbf{There is only one 2-point in the double chambers of $M$:} As there can be no multiple edges between 0- and 2-vertices, we can find a plane graph that has the same configuration of double chambers, a contradiction.
		
		\item \textbf{There are two different 2-points in the double chambers of $M$:} Let $M_f$ and $M_g$ be the subsets of $M$ containing the double chambers with 2-points $f$ and $g$ respectively. In the double chambers of each subset there are exactly two edges of $c$, as there cannot be fewer.
		
		Each of the two subsets is of one of the following three forms: 
		\begin{enumerate}[(a)]
			\item One double chamber that contains two edges of $c$
			\item Two double chambers sharing a 1-side
			\item Two double chambers not sharing a 1-side
		\end{enumerate}
		
		Assume first that one of the sets, say w.l.o.g. $M_f$ is of form $(c)$. As there can be no edge of $c$ between the two 0-points of a double chamber, the double chambers of $M_f$ must share their 2-sides with double chambers of $M_g$, so $M_g$ cannot be of form $(a)$. If $M_g$ would be of form $(b)$ then there would be a double edge between $f$ and a 0-vertex, a contradiction. It follows that if $M_f$ or $M_g$ is of form $(c)$, then the other set is also of form $(c)$. Two non-adjacent vertices of $c$ are 2-points and the other 2 vertices of $c$ are on different 2-sides. Assume from now on that neither $M_f$ nor $M_g$ are of form $(c)$. We will prove that this leads to a contradiction.
		
		Assume w.l.o.g. that $M_f$ is of form $(a)$. Then $M_g$ cannot be of form $(a)$ for then any plane graph would have the same configuration of double chambers or there would be a double edge in $G$. It follows that $M_g$ is of form $(b)$. If the double chamber in $M_f$ would share its 2-side with a double chamber of $M_g$ we could again find the same configuration in a plane graph, so the double chamber in $M_f$ shares only its 0-points with double chambers in $M_g$. As there are no multiple edges in $G$, these 0-points are the two 0-points that are not shared by the two double chambers of $M_g$. Now the edges of $c$ in $M_g$ are in contradiction with Lemma~\ref{lem:intersecting_edges} because two of the vertices of $c$ in $M_g$ are 0-points and the other is not a 2-point.
		
		We can now assume that both $M_f$ and $M_g$ are of form $(b)$. If the double chambers in $M_f$ and $M_g$ would not share any 2-sides, then exactly as in the previous case we get a contradiction with Lemma~\ref{lem:intersecting_edges}. If the double chambers in $M_f$ and $M_g$ would share both 2-sides then $G$ would have a vertex of degree 2 which is not the case. It follows that they share exactly one 2-side. The remaining 0-points are also shared as otherwise we could find a plane graph with this configuration of double chambers. This shared 0-point is in $c$. 
		Let $D_1$ be the double chamber in $M_f$ that shares its 2-side with a double chamber in $M_g$, and let $D_2$ be the other double chamber in $M_f$. There is an edge $e$ in $c$ between the 0-point of $D_2$ that is not in $D_1$ and a vertex $x$ on the 1-side shared by $D_1$ and $D_2$. The vertex $x$ is not a 2-point. There is also an edge of $c$ between $x$ and a vertex $y$ on the 2-side of $D_1$. As there is an edge $e'$ in $D_1$ with $\pi(e')=\pi(e)$, Lemma~\ref{lem:intersecting_edges} implies that $y$ is the 0-point that is shared by $D_1$ and $D_2$. But then $\{x,y\}$ is the only edge in $D_1$ and it has both its vertices on the same 1-side. Lemma~\ref{lem:minimal_path} implies that $\{x,y\}$ is also in $D_2$ which is in contradiction with the minimality of $M$.
	\end{itemize}
\end{proof}

\subsection*{Breaking vertices}
We will now formally define what it means for a set of vertices of $O(G)$ to `break' a vertex of $G$. One of the reasons that Dual does not preserve 3-connectivity is because two vertices of $G^*$ can break too many vertices of $G$. 

\begin{definition}\label{def:break_vertex}
	Let $G$ be an embedded graph and let $O$ be a lopsp-operation. A set $X$ of vertices of $O(G)$ \emph{breaks a vertex}\index{break!a vertex} $v$ of $G$ if all of the vertices of $S_O(v)$ are in $X$, or if not all vertices of $S_O(v)\setminus X$ are in the same component of $O(G)\setminus X$. 
\end{definition}

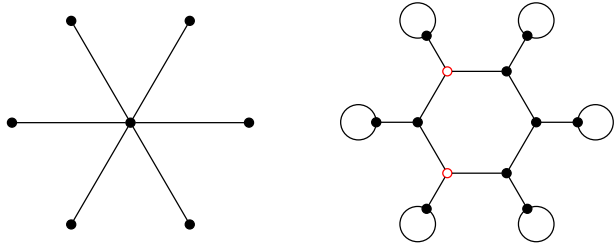
\begin{figure}
	\centering
	\scalebox{1.2}{\begin{tikzpicture}[scale=1.5, rotate = 90]
\tikzset{type0/.style={shape=circle, draw=black, scale=0.3, fill=black}}
\tikzset{type1/.style={shape=circle, draw=black!30!green, scale=0.3, fill=white}}
\tikzset{type2/.style={shape=circle, draw=black, scale=0.3, fill=white}}
\tikzset{t0/.style={draw=red, thin}}
\tikzset{t1/.style={draw=black!40!green, thin}}
\tikzset{t2/.style={}}



\node[type0] (v1) at ({(cos(120)-1)/2}, {sin(60)/2}) {};
\node[type0] (v2) at (0, {sin(60)}) {};
\node[type0] (v3) at({(1 - cos(120))/2}, {sin(60)/2}) {};
\node[type0] (v4) at({(1 - cos(120))/2}, {-sin(60)/2}) {};
\node[type0] (v5) at({(0}, {-sin(60)}) {};
\node[type0] (v6) at ({(cos(120)-1)/2}, {-sin(60)/2}) {};

\node[type0] (f) at (0, 0) {};

\begin{scope}[t2]
\draw (f) edge (v1);
\draw (f) edge (v2);
\draw (f) edge (v3);
\draw (f) edge (v4);
\draw (f) edge (v5);
\draw (f) edge (v6);
\end{scope}

\draw[draw=none]  ({(cos(120)-1)/2}, {sin(60)/2}) circle [radius = {0.15*sin(60)}];
\draw[draw=none]  (0, {sin(60)}) circle [radius = {0.15*sin(60)}];
\draw[draw=none] ({(1 - cos(120))/2}, {sin(60)/2}) circle [radius = {0.15*sin(60)}];
\draw[draw=none] ({(1 - cos(120))/2}, {-sin(60)/2}) circle [radius = {0.15*sin(60)}];
\draw[draw=none] ({(0}, {-sin(60)}) circle [radius = {0.15*sin(60)}];
\draw[draw=none]  ({(cos(120)-1)/2}, {-sin(60)/2}) circle [radius = {0.15*sin(60)}];

\end{tikzpicture}} \qquad
	\scalebox{1.2}{\begin{tikzpicture}[scale=1.5, rotate = 90]
\tikzset{type0/.style={shape=circle, draw=black, scale=0.3, fill=black}}
\tikzset{type1/.style={shape=circle, draw=black!30!green, scale=0.3, fill=white}}
\tikzset{red/.style={shape=circle, draw=red, scale=0.3, fill=white}}
\tikzset{t0/.style={draw=red, thin}}
\tikzset{t1/.style={draw=black!40!green, thin}}
\tikzset{t2/.style={}}

\node[red] (v1') at ({0.5*(cos(120)-1)/2}, {0.5*sin(60)/2}) {};
\node[type0] (v2') at (0, {0.5*sin(60)}) {};
\node[red] (v3') at({(1 - cos(120))/4}, {sin(60)/4}) {};
\node[type0] (v4') at ({(1 - cos(120))/4}, {-sin(60)/4}) {};
\node[type0] (v5') at (0, {-0.5*sin(60)}) {};
\node[type0] (v6') at ({(cos(120)-1)/4}, {-sin(60)/4}) {};

\node[type0] (v1) at ({0.85*(cos(120)-1)/2}, {0.85*sin(60)/2}) {};
\node[type0] (v2) at (0, {0.85*sin(60)}) {};
\node[type0] (v3) at({0.85*(1 - cos(120))/2}, {0.85*sin(60)/2}) {};
\node[type0] (v4) at({0.85*(1 - cos(120))/2}, {-0.85*sin(60)/2}) {};
\node[type0] (v5) at({(0}, {-0.85*sin(60)}) {};
\node[type0] (v6) at ({0.85*(cos(120)-1)/2}, {-0.85*sin(60)/2}) {};

\begin{scope}[t2]
	\draw (v1) edge (v1');
	\draw (v2) edge (v2');
	\draw (v3) edge (v3');
	\draw (v4) edge (v4');
	\draw (v5) edge (v5');
	\draw (v6) edge (v6');
	
	\draw  ({(cos(120)-1)/2}, {sin(60)/2}) circle [radius = {0.15*sin(60)}];
	\draw  (0, {sin(60)}) circle [radius = {0.15*sin(60)}];
	\draw ({(1 - cos(120))/2}, {sin(60)/2}) circle [radius = {0.15*sin(60)}];
	\draw ({(1 - cos(120))/2}, {-sin(60)/2}) circle [radius = {0.15*sin(60)}];
	\draw ({(0}, {-sin(60)}) circle [radius = {0.15*sin(60)}];
	\draw  ({(cos(120)-1)/2}, {-sin(60)/2}) circle [radius = {0.15*sin(60)}];
	
	\draw (v1')--(v2') -- (v3') --(v4') --(v5') --(v6') -- (v1');
	
\end{scope}

\end{tikzpicture}}
	\caption{On the left an embedding of the star graph $K_{1,6}$ is shown. Let $v$ be the middle vertex. In the right image Truncation is applied to the embedded graph. The six vertices in the middle form the vertex-shadow of $v$. The set of the two red vertices breaks $v$.}
	\label{fig:vertex_image}
\end{figure}

	Figure~\ref{fig:vertex_image} shows an example of a set of two vertices in a vertex-shadow that may break the corresponding vertex.

\begin{lemma}\label{lem:verteximage_cycle}
	Let $O$ be a $c3$-lopsp-operation different from Dual, and let $G$ be a simple embedded graph. If $v_0$ is of type 2, then the vertices of any vertex-shadow form a (simple) cycle, and the intersection of two vertex-shadows is either empty, one vertex, or one edge.
\end{lemma}
\begin{proof}
	Let $P$ be a minimal cut-path of $O$ and assume that there is a vertex $v$ in $G$ such that the face corresponding to $v$ is not a cycle. This is the case if and only if there is a cycle of length two in $B_{O(G)}$ with $v$ as one of its vertices. Let $x$ be the other vertex. The cycle is completely contained in the double chambers with 0-point $v$, as each of the edges is incident with $v$. By Lemma~\ref{lem:dual_char}, $x$ is not a 2-point. As there are no multiple edges in $G$, this implies that either both edges of the cycle are in the same double chamber, or the edges are in double chambers sharing a 1- or 2-side. In any case Lemma~\ref{lem:c3char} implies that $O$ is not $c3$, a contradiction. It follows that the vertices of $S_O(v)$ form a cycle.
	
	Assume that a vertex $x$ is in two different vertex-shadows $S_O(v)$ and $S_O(w)$. First assume that $v$ and $w$ are not adjacent in $G$. It is clear that $\pi(x)$ is not $v_2$ or $v_0$, so $x$ is on the 1-side shared by a double chamber containing $v$ and a double chamber containing $w$. It follows from Lemma~\ref{lem:intersecting_edges} with $\pi(x)$ and $v_0$ as the vertices of the cut-path that the edges $\{x,v\}$ and $\{x,w\}$ cannot both exist. It follows that $v$ and $w$ are adjacent vertices of $G$.

	As $G$ has no multiple edges, there are exactly two double chambers that contain both $v$ and $w$. Those double chambers form the diamond around the edge $\{v,w\}$. The vertex $x$ is in that diamond. It follows that $\{v,x\}$ and $\{w,x\}$ are also in the diamond: If $x$ is not on a 1-side this is trivial, if $x$ is on a 1-side this follows by Lemma~\ref{lem:minimal_path}.
	Assume that there is another vertex $y$ in both $S_O(v)$ and $S_O(w)$. The edges $\{y,v\}$ and $\{y,w\}$ are also in the diamond around $\{v,w\}$. It now follows that the edges $\{v,x\}$, $\{x,w\}$, $\{w,y\}$ and $\{y,v\}$ induce a 4-cycle in the $P$-diamond of $O$. By Definition~\ref{def:c3} this cycle must be trivial. It follows that $x$ and $y$ are adjacent in $O(G)$ and the edge between them is both in face $v$ and face $w$.
	
	Now assume that $S_O(v)$ and $S_O(w)$ share more than the edge $\{x,y\}$. Then in the diamond around $\{v,w\}$ there is another vertex $z$ of type 0 or 1 that is not the vertex corresponding to the edge $\{x,y\}$ and there are also edges $\{z,v\}$ and $\{z,w\}$. Now there are three 4-cycles in the diamond around $\{v,w\}$: one containing $x$ and $y$, one containing $y$ and $z$, and one containing $x$ and $z$. All of these 4-cycles share $v$ and $w$. As $z$ is not the vertex corresponding to the edge $\{x,y\}$, at least one of the cycles is non-trivial, which is a contradiction with Lemma~\ref{lem:c3char}.
\end{proof}

\begin{corollary}\label{cor:breaking_vertices}
	Let $O$ be a $c3$-lopsp-operation different from Dual, and let $G$ be a simple embedded graph whose vertices all have degree at least three. A set of two vertices of $O(G)$ can break at most two vertices of $G$. If $v_0$ is of type 2 then such a set can break at most one vertex. If $k\in\{1,2\}$ vertices are broken, at least $k$ vertices of the breaking set are in the vertex-shadows of the broken vertices.
\end{corollary}
\begin{proof}
	If $v_0$ is of type 0 then every vertex-shadow consists of only one vertex and a vertex of $O(G)$ can be in at most one vertex-shadow. In this case the statement follows immediately.
	
	Assume that $v_0$ is of type 2. As there are no vertices of degree 1 or 2 in $G$, every vertex-shadow contains at least three vertices, so not every vertex of a vertex-shadow can be in the breaking set. By Lemma~\ref{lem:verteximage_cycle} the vertices of every vertex-shadow form a cycle, which implies that every broken vertex has at least two non-adjacent vertices of the breaking set in its vertex-shadow. If there would be two broken vertices then they would have two non-adjacent vertices in common, which contradicts Lemma~\ref{lem:verteximage_cycle}. It follows that if $v_0$ is of type 2, two vertices of $O(G)$ can break at most one vertex of $G$.
\end{proof}

\subsection*{Breaking edges}

To preserve the 3-connectivity of a graph, it is not sufficient that its vertices are not broken too easily. The structure of the edges between them must also be preserved. The next definition describes what it means for vertices in $O(G)$ to `break' edges of $G$. 

\begin{definition}\label{def:edge_breaking}
	For an embedded graph $G$ and a lopsp-operation $O$, let $X$ be a set of vertices of $O(G)$. 
	We say that $X$ \emph{breaks the edge}\index{break!an edge} $\{v,w\}$ if there are no vertices $a\in S_O(v)$ and $b\in S_O(w)$ that are in the same component of $O(G)\setminus X$, i.e. there are no paths in $O(G)\setminus X$ between vertices of $S_O(v)$ and $S_O(w)$.
	If either $S_O(v)$ or $S_O(w)$ only contains vertices of $X$, then $X$ also breaks $\{v,w\}$. 
	
	A lopsp-operation $O$ is \emph{edge-breaking} if $v_2$ and $v_1$ are adjacent vertices in $O$ and $v_2$ is of type 0. It follows that $v_1$ is of type 1 or type 2. Both these possibilities are shown in Figure \ref{fig:edge_breakers}. Edge-breaking operations with a $v_1$-vertex of type $i$ are said to be of \emph{type} $i$.
	A lopsp-operation that is not edge-breaking is \emph{edge-preserving}.
	
	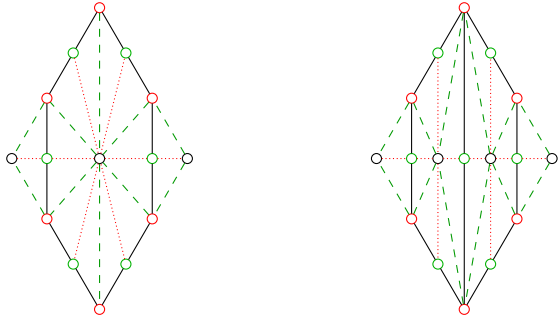
\begin{figure}	
			\centering
			\begin{tikzpicture}
\tikzset{type0/.style={shape=circle, draw=red, scale=0.4, fill=white}}
\tikzset{type1/.style={shape=circle, draw=black!30!green, scale=0.4, fill=white}}
\tikzset{type2/.style={shape=circle, draw=black, scale=0.4, fill=white}}
\tikzset{0edge/.style={draw=red, densely dotted}}
\tikzset{1edge/.style={draw=black!40!green, dashed}}
\tikzset{2edge/.style={}}

\node[type2] (v0) at (-{2*tan(30},0) {};
\node[type2] (v0') at ({2*tan(30},0) {};
\node[type2] (v1) at (0,0) {};
\node[type0] (v2') at (0,-2) {};
\node[type0] (v2) at (0,2) {};

\node[type1] (e1) at ({1.2*tan(30},0) {};
\node[type1] (e1') at ({-1.2*tan(30},0) {};
\node[type0] (x') at ({-1.2*tan(30},  0.8) {};
\node[type0] (x'') at ({-1.2*tan(30}, -0.8) {};
\node[type0] (x''') at ({1.2*tan(30},  -0.8) {};
\node[type0] (x) at ({1.2*tan(30},  0.8) {};
\node[type1] (e') at ({-0.6*tan(30}, 1.4) {};
\node[type1] (e) at ({0.6*tan(30}, 1.4) {};
\node[type1] (e'') at ({-0.6*tan(30}, -1.4) {};
\node[type1] (e''') at ({0.6*tan(30}, -1.4) {};

\begin{scope}[2edge]
\draw (e1) -- (x) -- (e) -- (v2) -- (e') -- (x') -- (e1') -- (x'') -- (e'') -- (v2') -- (e''') -- (x''') -- (e1);
\end{scope}

\begin{scope}[1edge]
\draw (v0') -- (x) -- (v1) -- (x') -- (v0);
\draw (v0') -- (x''') -- (v1) -- (x'') -- (v0);
\draw (v2) -- (v1) -- (v2'); 

\end{scope}

\begin{scope}[0edge]
\draw (e') -- (v1) -- (e''); 
\draw (e''') -- (v1) -- (e); 
\draw (v0) -- (e1') -- (v1) -- (e1) -- (v0');
\end{scope}

\end{tikzpicture}
			\qquad \qquad \qquad
			\begin{tikzpicture}
\tikzset{type0/.style={shape=circle, draw=red, scale=0.4, fill=white}}
\tikzset{type1/.style={shape=circle, draw=black!30!green, scale=0.4, fill=white}}
\tikzset{type2/.style={shape=circle, draw=black, scale=0.4, fill=white}}
\tikzset{0edge/.style={draw=red, densely dotted}}
\tikzset{1edge/.style={draw=black!40!green, dashed}}
\tikzset{2edge/.style={}}

\node[type2] (v0) at (-{2*tan(30},0) {};
\node[type2] (v0') at ({2*tan(30},0) {};
\node[type1] (v1) at (0,0) {};
\node[type0] (v2') at (0,-2) {};
\node[type0] (v2) at (0,2) {};

\node[type1] (e1) at ({1.2*tan(30},0) {};
\node[type1] (e1') at ({-1.2*tan(30},0) {};
\node[type0] (x') at ({-1.2*tan(30},  0.8) {};
\node[type0] (x'') at ({-1.2*tan(30}, -0.8) {};
\node[type0] (x''') at ({1.2*tan(30},  -0.8) {};
\node[type0] (x) at ({1.2*tan(30},  0.8) {};
\node[type1] (e') at ({-0.6*tan(30}, 1.4) {};
\node[type1] (e) at ({0.6*tan(30}, 1.4) {};
\node[type1] (e'') at ({-0.6*tan(30}, -1.4) {};
\node[type1] (e''') at ({0.6*tan(30}, -1.4) {};
\node[type2] (f) at ({0.6*tan(30}, 0) {};
\node[type2] (f') at ({-0.6*tan(30}, 0) {};

\begin{scope}[2edge]
\draw (e1) -- (x) -- (e) -- (v2) -- (e') -- (x') -- (e1') -- (x'') -- (e'') -- (v2') -- (e''') -- (x''') -- (e1);
\draw (v2) -- (v1) -- (v2'); 
\end{scope}

\begin{scope}[1edge]
\draw (v0') -- (x) -- (f) -- (v2) -- (f') -- (x') -- (v0) ;
\draw (v0')-- (x''') -- (f) -- (v2') -- (f') -- (x'') -- (v0);
\end{scope}

\begin{scope}[0edge]
\draw (v0)-- (e1') -- (f')
		-- (v1) -- (f) -- (e1) -- (v0') ;
\draw (e') -- (f') -- (e''); 
\draw (e''') -- (f) -- (e); 
\end{scope}

\end{tikzpicture}
		\caption{This figure shows diamonds of two different edge-breaking operations. The left operation is of type 2 and the right one is of type 1.}
		\label{fig:edge_breakers}
	\end{figure}
\end{definition}

\begin{remark*}\label{rem:half_12}
	Except for Dual, edge-breaking operations always exist in pairs. For every edge-breaking operation $O_2$ of type 2, replacing the two copies of $\{v_1,v_2\}$ in a diamond of $O_2$ by a 4-cycle of type 1 with a vertex of type 1 in its interior gives an edge-breaking operation $O_1$ of type 1. Conversely for every edge-breaking operation of type 1 there is such a 4-cycle with $v_1$ in its interior that can be replaced by two edges to get an edge-breaking operation of type 2. Such a pair of operations $O_1$ and $O_2$ is illustrated in Figure~\ref{fig:edge_breakers}. Note that for any embedded graph $G$, $O_1(G)$ is the union of $O_2(G)$ and the dual of $G$.
\end{remark*}

The following lemmas explain the name `edge-breaking'. A \emph{type-$i$ walk} or \emph{type-$i$ path} is respectively a walk or path that only has edges of type $i$. A type-2 walk or path in $B_{O(G)}$ with endpoints of type 0 corresponds to a walk or path in $O(G)$.

\begin{definition}\label{def:edge-path}
	Let $O$ be a lopsp-operation with a cut-path $P$. 
	Consider the path $P_{v_{0,L},v_{1}}\cup P_{v_{1},v_{0,R}}$ in $O_P$. We remove $v_{0,L}$ and $v_{0,R}$ if those are of type 2. Then we replace every vertex $v$ of type 2 in the path by the type-2 walk along all its neighbours in the order of their appearance in the cyclic order around the vertex $v$. This uniquely defines a walk in $O_P$, which is the \emph{shadow-connecting walk} for $O$ with cut-path $P$.
	
	An \emph{edge-path} for $O$ is a type-2 path in a diamond of $O$ between vertices in the two different vertex-shadows that only contains vertices that are in at least one of the two copies of the shadow-connecting walk.

\end{definition}

\begin{lemma}\label{lem:edge-path}
	Let $O$ be a $c3$-lopsp-operation. 
	\begin{enumerate}
		\item For all cut-paths $P$ in $O$ there is an edge-path in the $P$-diamond of $O$.
		\item $O$ is edge-preserving if and only if for all cut-paths $P$ there exists an edge-path that does not contain either of the 2-points of the diamond.
	\end{enumerate}
\end{lemma}
\begin{proof}
	\begin{enumerate}
		\item Let $P$ be a cut-path of $O$. 
		If the first and last vertex of the shadow-connecting walk are of type 0 then the general statement follows immediately. If the first (or similarly the last) vertex $x$ of the shadow-connecting walk is of type 1, then the second one $y$ is of type 0, it is in the diamond and it is adjacent to the same 0-point as $x$. It follows that if we remove the first and the last vertex of $Q$, the shadow-connecting walk induces a type-2 walk in the $P$-diamond from a vertex in $N_0(v_{0,L})$ to a vertex in $N_0(v_{0,R})$. This implies the existence of a path with the required properties.
		
		\item Let $P$ be any cut-path.
		Assume first that $O$ is edge-breaking. An edge-path must contain one of the 2-points or $v_1$ by the definition of an edge-breaking operation and the Jordan curve theorem. If an edge-path does not contain the 2-points then $v_1$ is in the edge-path. This implies that $v_1$ is of type 1. The endpoints of an edge-path are of type 0, so the neighbours of type 0 of $v_1$ must also be in the edge-path. As $O$ is edge-breaking, these neighbours are the 2-points. It follows that for an edge-breaking operation, every edge-path contains at least one of the 2-points of the diamond, which proves one implication.
		
		Conversely, assume that $O$ is edge-preserving. Let $D$ and $D'$ be the two copies of $O_P$ in the diamond.
		The shadow-connecting walk $Q$ induces a walk $Q_{D}$ in $D$ and a walk $Q_{D'}$ in $D'$. Let $v_2^D$ be the 2-point of $D$ and let $v_2^{D'}$ be the 2-point of $D'$. If $v_2^D$ is not in $Q_D$ we are done. Assume that $v_2^D$ is in $Q_D$. By definition of a shadow-connecting walk, $v_2^D$ is adjacent to a 2-vertex $v$ in $D\cap D'$. This vertex $v$ is not $v_1$, as $v_2$ is of type 0 and $O$ is edge-preserving. It is also not $v_0$, for then $v_2\in N_0(v_0)$ and by Lemma~\ref{lem:dual_char} $O$ would be the edge-breaking operation Dual. It follows that $\pi(v)\neq v_0,v_1$. 
		
		Let $x$ and $y$ be the neighbours of $v$ in $D\cap D'$ and let $v'\neq v$ be the unique vertex in $D\cap D'$ such that $\pi(v)=\pi(v')$. Note that $x$ and $y$ are in both $Q_D$ and $Q_{D'}$ by definition of a shadow-connecting walk. Assume that the segment of $Q_{D'}$ between $y$ and $x$ contains $v_2^{D'}$. Then there is an edge $\{v,v_2^{D'}\}$ and by symmetry there are also edges $\{v',v_2^{D}\}$ and $\{v',v_2^{D'}\}$. It follows that $v,v_2^D,v',v_2^{D'}$ is a 4-cycle in two adjacent copies of $O_P$ that has $v_1$ in its interior. By Lemma~\ref{lem:c3char}, this cycle must be trivial, which implies that $v_1$ is of type 1 and it is adjacent to $v_2$. This is a contradiction as $v_2$ is of type 0 and $O$ is edge-preserving. It follows that the segment of $Q_{D'}$ between $x$ and $y$ does not contain $v_2^{D'}$. Now we can replace the segment of $Q_D$ between $x$ and $y$ by the segment of $Q_{D'}$. Doing such a replacement for every appearance of $v_2^D$ in $Q_D$ and removing the first or last vertex if they are of type 1, we get a type-2 walk in the $P$-diamond of $O$ between a vertex in $N_0(v_{0,L})$ and a vertex in $N_0(v_{0,R})$ that does not contain 2-points. The lemma follows.
	\end{enumerate}
\end{proof}

\begin{lemma}\label{lem:2_broken}
	Let $O$ be a $c3$-lopsp-operation different from Dual and let $G$ be a simple embedded graph whose vertices all have degree at least three. Furthermore, assume that $O$ is edge-preserving or that $G^*$ is simple. Let $X$ be a set of two vertices of $O(G)$. If $X$ breaks $k$ vertices, then it breaks at most $2-k$ edges that are not incident with one of the broken vertices.
\end{lemma}
\begin{proof}
	First we prove that a vertex $x$ in a vertex-shadow $S_O(v)$ cannot be in a shadow-connecting walk in a double chamber with 0-points different from $v$. 
	
	For $x$ to be in $S_O(v)$, it must be in a double chamber $D$ with 0-point $v$. Assume that $x$ is also in a shadow-connecting walk $Q_{D'}$ in a double chamber $D'$ with 0-points $u\neq v$ and $w\neq v$. As $O$ is not Dual and $x$ is in a vertex-shadow, $x$ is not a 2-point. As there are no edges between the two 0-points of a double chamber, $x$ is also not a 0-point. Our vertex $x$ is in at least two different double chambers -- $D$ and $D'$ -- with different 1-points. It follows that it is on the shared 1-side of $D$ and $D'$, but it is not one of the points of that side. Assume w.l.o.g.\ that $u$ is the 0-point of the shared 1-side. This situation in shown in the left image of Figure~\ref{fig:2_broken}. As $x$ is in $S_O(v)$, there is an edge in $D$ between $v$ and $x$ which is on the opposite 1-side. It now follows from Lemma~\ref{lem:intersecting_edges} that any path from $x$ to $w$ in $D'$ also contains $u$ or the vertex $x'\neq x$ in $D'$ with $\pi(x')=\pi(x)$. As $Q_{D'}$ is a shadow-connecting walk there is an edge in $D'$ between $x$ and a 2-vertex $f\neq u$ on the 2-side of $D'$. As shown in the right image of Figure~\ref{fig:2_broken}, adding to $\{x,f\}$ only edges of $B_{O(G)}$ on the 2-side of $D'$ we get a path in $D'$ between $x$ and $w$ that does not contain $u$ or $x'$. This is a contradiction. 
	
	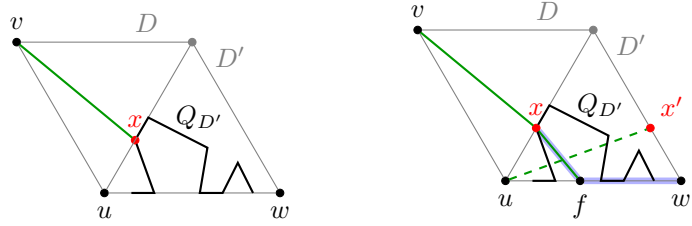
\begin{figure}		
		\begin{center}
			\begin{tikzpicture}[]
\tikzset{bluevertex/.style={shape=circle, draw=blue, scale=0.3, fill=blue}}
\tikzset{grayvertex/.style={shape=circle, draw=gray, scale=0.3, fill=gray}}
\tikzset{redvertex/.style={shape=circle, draw=red, scale=0.3, fill=red}}
\tikzset{blackvertex/.style={shape=circle, draw=black, scale=0.3, fill=black}}
\tikzset{blueedge/.style={draw=blue, thick}}
\tikzset{greenedge/.style={draw=black!40!green, thick}}
\tikzset{grayedge/.style={shape=circle, draw=gray}}

\tikzset{noNode/.style={draw=none}}

\node[gray] (D') at (0.5,1.8) {$D'$};
\node[gray] (D) at (-0.6,2.25) {$D$};
\node[black] (Q) at (0.1,1) {$Q_{D'}$};
\node[grayvertex] (v2) at (0,2) {};
\node[blackvertex,label={[label distance=-0.0cm]-90:{$\color{black}{w}$}}] (w) at ({2*tan(30)}, 0) {};
\node[blackvertex,label={[label distance=-0.0cm]-90:{$\color{black}{u}$}}] (u) at ({-2*tan(30)}, 0) {};

\node[blackvertex,label={[label distance=-0.0cm]90:{$\color{black}{v}$}}] (v) at ({-4*tan(30)}, 2) {};

\node[redvertex,label={[label distance=-0.0cm]90:{$\color{red}{x}$}}] (x) at ({-1.3*tan(30)}, 0.7) {};

\begin{scope}[grayedge]
\draw (v) -- (u) --(w);
\draw  (v) -- (v2) --(w) ;
\draw (v2) -- (u);
\end{scope}

\begin{scope}[greenedge]
\draw (v) -- (x);

\end{scope}

\begin{scope}[draw=black, thick]
\draw (-0.8,0) --(-0.5,0) -- (x) -- ({-1*tan(30)}, 1) -- (0.2,0.6) -- (0.1,0) -- (0.4,0) -- (0.6,0.4) -- (0.8,0)  ;
\end{scope}
\end{tikzpicture} \hspace{1cm}
			\begin{tikzpicture}[]
\tikzset{bluevertex/.style={shape=circle, draw=blue, scale=0.3, fill=blue}}
\tikzset{grayvertex/.style={shape=circle, draw=gray, scale=0.3, fill=gray}}
\tikzset{redvertex/.style={shape=circle, draw=red, scale=0.3, fill=red}}
\tikzset{blackvertex/.style={shape=circle, draw=black, scale=0.3, fill=black}}
\tikzset{blueedge/.style={draw=blue, thick}}
\tikzset{greenedge/.style={draw=black!40!green, thick}}
\tikzset{grayedge/.style={shape=circle, draw=gray}}

\tikzset{noNode/.style={draw=none}}

\node[gray] (D') at (0.5,1.8) {$D'$};
\node[gray] (D) at (-0.6,2.25) {$D$};
\node[black] (Q) at (0.1,1) {$Q_{D'}$};
\node[grayvertex] (v2) at (0,2) {};
\node[blackvertex,label={[label distance=-0.0cm]-90:{$\color{black}{w}$}}] (w) at ({2*tan(30)}, 0) {};
\node[blackvertex,label={[label distance=-0.0cm]-90:{$\color{black}{u}$}}] (u) at ({-2*tan(30)}, 0) {};

\node[blackvertex,label={[label distance=-0.0cm]90:{$\color{black}{v}$}}] (v) at ({-4*tan(30)}, 2) {};
\node[blackvertex,label={[label distance=-0.0cm]-90:{$\color{black}{f}$}}] (f) at ({-0.3*tan(30)}, 0) {};
\node[redvertex,label={[label distance=-0.0cm]90:{$\color{red}{x}$}}] (x) at ({-1.3*tan(30)}, 0.7) {};
\node[redvertex,label={[label distance=-0.02cm]70:{$\color{red}{x'}$}}] (x') at ({1.3*tan(30)}, 0.7) {};

\draw[draw=blue, line width=2.5pt, opacity=0.3] (x) -- (f) -- (w);

\begin{scope}[grayedge]
\draw (v) -- (u) --(f) --(w);
\draw  (v) -- (v2) --(x') --(w) ;
\draw (v2) --(x) -- (u);
\end{scope}

\begin{scope}[greenedge]
\draw (v) -- (x) --(f);
\draw[dashed] (u) -- (x');
\end{scope}

\begin{scope}[draw=black, thick]
\draw (-0.8,0) --(-0.5,0) -- (x) -- ({-1*tan(30)}, 1) -- (0.2,0.6) -- (0.1,0) -- (0.4,0) -- (0.6,0.4) -- (0.8,0)  ;
\end{scope}

\end{tikzpicture} 
		\end{center}
		\caption{These images show an example of the walks and vertices in the first part of the proof of Lemma~\ref{lem:2_broken}. The left image shows the initial assumptions, in the right one the path that leads to a contradiction is highlighted in blue.}
		\label{fig:2_broken}
	\end{figure}
		
	Let $k$ be the number of vertices that are broken by $X$. Let $X'$ be the set of vertices in $X$ that are not in the vertex-shadow of any of the broken vertices. By Corollary~\ref{cor:breaking_vertices} $k$ is at most 2, and $|X'|\leq 2-k$. Assume that $X$ breaks $3-k$ edges $e_1, \ldots, e_{3-k}$ of $G$ that are not incident with any of the broken vertices. Let $e_i=\{u_i,w_i\}$ for $i=1,\ldots, 3-k$. We will prove that $O$ is edge-breaking and that $G^*$ is not simple.
	
	Let $Q_i$ and $Q_i'$ be the two shadow-connecting walks in the diamond around $e_i$. As none of the $2(3-k)$ walks $Q_i$, $Q_i'$ ($i\in\{1,\ldots,3-k\}$) exist in $O(G)\setminus X$, each of those walks contains at least one vertex of $X$. We just proved that such a vertex cannot be in the vertex-shadow of any of the broken vertices so they must be in $X'$. If $k=2$ there are no vertices in $X'$ so they cannot be in any of those paths, a contradiction. Assume that $k<2$. If a vertex of $B_{O(G)}$ is not a 0- or 2-point, then it is in at most two double chambers. If all vertices in $X'$ would be such vertices then the vertices of $X'$ would be in at most $2(2-k)$ of the $2(3-k)$ double chambers, so at least two of our shadow-connecting walks would not contain a vertex of $X'$, a contradiction. It follows that at least one vertex of $X'$ is a 0- or 2-point. 
	
	A vertex of $X'$ cannot be a 0-point for then it would be the only vertex in a vertex-shadow and therefore the corresponding vertex would be broken. Thus, at least one vertex of $X'$ is a 2-point. If $|S'|=1$ this is the only vertex of $X'$, if $|X'|=2$ there is one other vertex. If that vertex is not a 2-point it can be in at most two of our shadow-connecting walks.
	
	In any case, there is at least one of the $3-k$ edges $e_i$, say $e_1$, such that the only vertices of $X'$ that are in $Q_1$ and $Q_1'$ are 2-points. It follows that there is either an edge $e_i$ that has the same face of $G$ on both sides, or if $|S'|=2$ the two vertices of $X'$ may correspond to faces in $G$ that share at least three edges. In either case, $G^*$ is not simple.
	
	If $O$ would be edge-preserving, then by Lemma \ref{lem:edge-path} there would be a path from a vertex in $S_O(u_1)$ to a vertex in $S_O(w_1)$ that only contains vertices and edges from $Q_1$ and $Q_1'$ and that does not contain a 2-point. As these 2-points are the only vertices of $X'$ in $Q_1$ and $Q'_1$, the path also exists in $O(G)\setminus X$, a contradiction. It follows that $O$ is edge-breaking.
\end{proof}

\subsection*{Local cuts}

To prove that edge-preserving operations preserve 3-connectivity we also need to prove that there can be no `local cuts' in $O(G)$. More specifically, every component of $O(G)\setminus X$ for a set $X$ of two vertices must contain all vertices of at least one vertex-shadow.

To reduce the number of cases that we need to check in Lemma~\ref{lem:diff_v0sets} and Lemma~\ref{lem:no_2_in_v0set} we will use Lemma~\ref{lem:type2_cutpath} to choose convenient cut-paths. Lemma~\ref{lem:remove_induced_subgraph} and Lemma~\ref{lem:path_is_induced} will be used to prove that we can always choose such cut-paths.

\begin{lemma}\label{lem:remove_induced_subgraph}
	Let $G$ be a plane triangulation with a connected embedded induced subgraph $H$. For any face $f$ of $H$, the vertices of $G$ in the interior of $f$ are in the same component of $G\setminus H$. If $H$ has no faces, i.e.\ it is just one vertex without edges, then $G\setminus H$ is connected.
\end{lemma}
\begin{proof}
	Let $x_0,x_1,\ldots, x_{n-1}$ be the order of the vertices in the face $f$, indices modulo $n$. It is possible that $x_i=x_j$ for $i\neq j$. For every $i\in\{0,\ldots,n-1\}$ let $y_{i,1}, y_{i,2},\ldots, y_{i,k_i}$ be the neighbours of $x_i$ in the interior of $f$ in the $i$-th angle, in the same order as they appear in the rotational order of $x_i$. As $G$ is a triangulation and $H$ is an induced subgraph, the vertices $y_{i,1}, y_{i,2},\ldots, y_{i,k_i}$ are different from the $x_i$ and they form a walk in the interior of $f$. 
	If an $x_i$ has no neighbour in the interior of $f$, then there is an edge between $x_{i-1}$ and $x_{i+1}$ as $G$ is a triangulation. Because $H$ is an induced subgraph that is not possible unless $f$ is an empty triangle, in which case the lemma holds. We can therefore assume that every $x_i$ has at least one neighbour in the interior of $f$.  
	
	Note that as $G$ is a triangulation, $y_{i,k_i}=y_{i+1,1}$ for every $i$. This allows to construct a closed walk $y_{0,1},\ldots, y_{0,k_0-1}, y_{1,1},\ldots, y_{1,k_1-1},\ldots, y_{n-1,1}, \ldots, y_{n-1,k_{n-1}-1}$ that contains all the vertices in the interior of $f$ that have a neighbour in $f$, and no vertices of $f$.
	By the definition of a bridge, there is a path from any other vertex in the interior of $f$ to a vertex $y_{i,j}$. This proves the first part of the lemma.
	
	If $H$ is only one vertex with no edges then there are no loops at that vertex as $H$ is induced. As $G$ is a triangulation the neighbours of the vertex of $H$ now form a closed walk in $G\setminus H$. This proves that the vertex of $H$ is not a cut-vertex of $G$, which proves the lemma. 
\end{proof}

\begin{lemma}\label{lem:path_is_induced}
	A type-2 path with endpoints of type 0 in a barycentric subdivision, lopsp-operation or double chamber patch $G$ is an induced subgraph of $G$. 
\end{lemma}
\begin{proof}
	Assume that such a path $P$ is not an induced subgraph. Then there is an edge $e\notin P$ whose vertices $x$ and $y$ are in $P$. As there are no edges between vertices of the same type, the vertices $x$ and $y$ are of types 0 and 1. Say w.l.o.g.\ that $t(x)=1$. Then $x$ has at most two incident edges of type 2 in $G$ because of the properties of barycentric subdivisions, lopsp-operations and double chamber patches. As $x$ is not one of the endpoints of $P$ all of its incident edges of type 2, including $e$, are in $P$, a contradiction.
\end{proof}

The idea in Lemma~\ref{lem:type2_cutpath} is that we would like to choose a cut-path that is of type 2. Generally this is not possible, but we can choose a cut-path such that either $P_{v_0,v_2}\setminus\{v_0,v_2\}$ is of type 2 or $P_{v_0,v_1}\setminus\{v_0\}$ is of type 2. Figure~\ref{fig:cutpaths} shows two double chamber patches that are obtained by cutting along cut-paths that satisfy the conditions in Lemma~\ref{lem:type2_cutpath}. A more concrete example is the cut-path of gyro shown in Figure~\ref{fig:lopsp_gyro}.

\begin{lemma}\label{lem:type2_cutpath}
	Let $O$ be a $c3$-lopsp-operation. Then there exists a cut-path $P$ such that at least one of the following is true:
	\begin{enumerate}[(a)]
		\item The path $P_{v_0,v_2}$ contains exactly one vertex $q_0$ in $N_0(v_0)$ and one vertex $q_2$ in $N_0(v_2)$. The subpath of $P_{v_0,v_2}$ between $q_0$ and $q_2$ is a type-2 path that contains all vertices of $P_{v_0,v_2}$ different from $v_0$ and $v_2$.
		
		\item The path $P_{v_0,v_1}$ contains exactly one vertex $q_0$ in $N_0(v_0)$. The subpath of $P_{v_0,v_1}$ between $q_0$ and $v_1$ is a type-2 path that contains all vertices of $P_{v_0,v_1}$ different from $v_0$. Let $q_1$ be the last vertex of type 0 in $P_{v_0,v_1}$. It is different from $v_0$ and it is the only vertex of $P_{v_0,v_1}$ that may be in $N_0(v_2)$.
		There is a type-2 path $R$ in $O$ between $q_1$ and a vertex $q_2$ in $N_0(v_2)$ that contains no other vertices of $P$ than $q_1$ and perhaps $v_2$. The only vertex of $N_0(v_2)$ in $R$ is $q_2$ and the only vertex of $R$ that may be in $N_0(v_0)$ is $q_1$.
		\end{enumerate}
\end{lemma}

\begin{figure}
	\begin{center}
		\scalebox{1}{\begin{tikzpicture}[scale= 1.2]
\tikzset{bluevertex/.style={shape=circle, draw=blue, scale=0.3, fill=blue}}
\tikzset{grayvertex/.style={shape=circle, draw=gray, scale=0.3, fill=gray}}
\tikzset{redvertex/.style={shape=circle, draw=red, scale=0.3, fill=red}}
\tikzset{blackvertex/.style={shape=circle, draw=black, scale=0.3, fill=black}}
\tikzset{blueedge/.style={draw=blue, thick}}
\tikzset{greenedge/.style={draw=black!40!green, thick}}
\tikzset{grayedge/.style={shape=circle, draw=gray}}

\tikzset{noNode/.style={draw=none}}

\node[grayvertex,label=90:$\color{gray}{v_2}$] (v2) at (0,2) {};
\node[blackvertex,label=-90:${q_{0,R}}$] (w) at ({2*tan(30)}, 0) {};
\node[blackvertex,label=-90:${q_{0,L}}$] (u) at ({-2*tan(30)}, 0) {};
\node[grayvertex] (v1) at (0, 0) {};
\node[blackvertex,label=0:${q_{2,R}}$] (q2) at ({tan(30)*0.3}, 1.7) {};
\node[blackvertex,label=180:${q_{2,L}}$] (q2') at ({-tan(30)*0.3}, 1.7) {};


\begin{scope}[grayedge]

\draw (q2) -- (v2) -- (q2');
\draw (u) --(v1);
\draw (v1)--(w);
\end{scope}

\begin{scope}[greenedge]

\end{scope}

\begin{scope}[draw=black, very thick]
\draw  (q2) --(w) ;
\draw (q2') -- (u);

\end{scope}

\end{tikzpicture} \hspace{1cm}}
		\scalebox{1}{\begin{tikzpicture}[scale= 1.2]
\tikzset{bluevertex/.style={shape=circle, draw=blue, scale=0.3, fill=blue}}
\tikzset{grayvertex/.style={shape=circle, draw=gray, scale=0.3, fill=gray}}
\tikzset{redvertex/.style={shape=circle, draw=red, scale=0.3, fill=red}}
\tikzset{blackvertex/.style={shape=circle, draw=black, scale=0.3, fill=black}}
\tikzset{blueedge/.style={draw=blue, thick}}
\tikzset{greenedge/.style={draw=black!40!green, thick}}
\tikzset{grayedge/.style={shape=circle, draw=gray}}

\tikzset{noNode/.style={draw=none}}

\node[blackvertex,label=90:$q_{2}$] (v2) at (0,2) {};
\node[blackvertex,label=-90:$q_{0,R}$] (w) at ({2*tan(30)}, 0) {};
\node[blackvertex,label=-90:$q_{0,L}$] (u) at ({-2*tan(30)}, 0) {};
\node[blackvertex,label=-90:$v_1$] (v1) at (0, 0) {};
\node[bluevertex] (x1) at (0.2, 0.5) {};
\node[bluevertex] (x2) at (-0.35, 0.8) {};
\node[bluevertex] (x3) at (0.1, 1.3) {};

\node[blackvertex,label={[label distance=0.08cm]90:{$\color{black}{q_{1,L}}$}}] (y) at (-0.3, 0) {};
\node[blackvertex,label={[label distance=-0.05cm]90:{$\color{black}{q_{1,R}}$}}] (y') at (0.3, 0) {};


\begin{scope}[grayedge]

\draw  (v2) --(w) ;
\draw (v2) -- (u);
\end{scope}

\begin{scope}[greenedge]

\end{scope}

\begin{scope}[draw=black, very thick]
\draw (u) --(y)--(v1);
\draw (v1)--(y') --(w);
\draw[blue] (y)--(x1)--(x2)--(x3)--(v2);
\end{scope}

\node[blue] (_) at (0.1,0.9) {$R$};
\end{tikzpicture} }
	\end{center}
	\caption{This figure shows a schematic representation of double chamber patches for cut-paths of the two different forms in Lemma~\ref{lem:type2_cutpath}. Only the type-2 paths that are described in that lemma are drawn in black or blue.}
	\label{fig:cutpaths}
\end{figure}
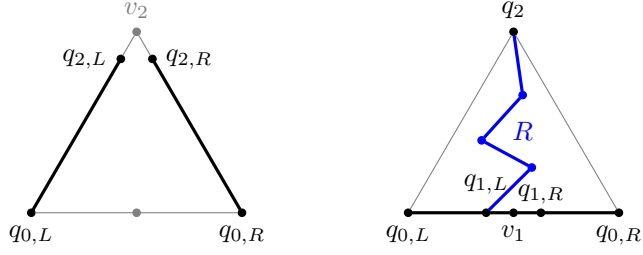
\begin{proof}	
	Consider any path in $O$ between a vertex in $N_0(v_0)$ and a vertex in $N_0(v_2)$. By replacing the type-2 vertices in this path by walks along their neighbours we get a type-2 walk between those same vertices which implies that there is also a type-2 path $Q$ between a vertex $q_0$ in $N_0(v_0)$ and a vertex $q_2$ in $N_0(v_2)$ that contains no other vertices of $N_0(v_0)$ or $N_0(v_2)$. Both $q_0$ and $q_2$ are of type 0 and they are identical to or adjacent to $v_0$ and $v_2$ respectively. By adding at most one edge at each end we get a path $\widetilde{Q}$ in $O$ from $v_0$ to $v_2$. There are two cases:
	
	\begin{itemize}
		\item[\fbox{$v_1\notin Q$:}] Let $Q'$ be the graph induced by $v_2$ and $Q\setminus v_0$. The path $Q$ is an induced graph by Lemma~\ref{lem:path_is_induced}. If $v_2$ is in $Q$ then $Q'$ is a path. However, if $v_2$ is of type 2 there may be more than one edge between $v_2$ and $q_2$. All of those edges are in $Q'$. If $v_1$ is in the interior of the same face of $Q'$ as $v_0$ --which is always the case if $v_2$ is in $Q$-- then by Lemma~\ref{lem:remove_induced_subgraph} there is a path in $O\setminus Q'$ between $v_1$ and $v_0$. Adding this path to $\widetilde{Q}$ we get a cut-path of $O$ which satisfies (a).
		
		Now assume that $v_1$ is not in the interior of the same face of $Q'$ as $v_0$. Then $v_1$ is in the interior of a face of $Q'$ of size two with vertices $v_2$ and $q_2$. The boundary $c$ of this face is a cycle and therefore a Jordan curve in the plane graph $O$. For any cut-path, the subpath between $v_1$ and $v_0$ would contain $q_2$ as $v_2$ cannot be in that path. This implies that $c$ induces a 4-cycle of type 1 in every diamond of $O$, as shown in Figure~\ref{fig:type2_cutpath_proof}. This 4-cycle has $v_1$ in its interior. By Lemma~\ref{lem:c3char} the cycle must be trivial, which implies that $v_1$ is of type 1 and has degree 2, its neighbours are $v_2$ and $q_2$. Let $P_1$ be the subpath of $\widetilde{Q}$ from $v_0$ to $q_2$ together with the unique edge between $q_2$ and $v_1$. This is a path from $v_0$ to $v_1$. We can now apply Lemma~\ref{lem:remove_induced_subgraph} to the path $P_1\setminus v_0$ to get a path $P_2$ from $v_0$ to $v_2$ that intersects $P_1$ only in the vertex $v_0$. The path $P:=P_1\cup P_2$ is a cut-path of $O$ that satisfies the properties of (b). The path $R$ is just the vertex $q_2$. If $q_2$ would be $v_0$ then $Q'$ would be just one vertex and we would have found a cut-path with the first set of properties, so $q_2\neq v_0$.
		
		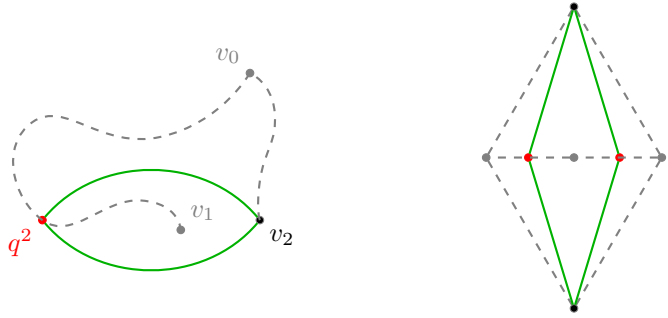
\begin{figure}		
			\begin{center}
				\begin{tikzpicture}[baseline=-35pt, scale=1.3]
\tikzset{bluevertex/.style={shape=circle, draw=blue, scale=0.3, fill=blue}}
\tikzset{grayvertex/.style={shape=circle, draw=gray, scale=0.3, fill=gray}}
\tikzset{redvertex/.style={shape=circle, draw=red, scale=0.3, fill=red}}
\tikzset{blackvertex/.style={shape=circle, draw=gray, scale=0.3, fill=black}}
\tikzset{blueedge/.style={shape=circle, draw=blue, thick}}
\tikzset{grayedge/.style={shape=circle, draw=gray, thick}}
\tikzset{greenedge/.style={shape=circle, draw=green!70!black, thick}}

\tikzset{noNode/.style={draw=none}}



\node[grayvertex,label={[label distance=-0.05cm]20:{$\color{gray}{v_1}$}}] (v1) at (0.2,-0.1) {};
\node[blackvertex,label={[label distance=-0.05cm]-20:{$v_2$}}] (v2) at (1,0) {};
\node[grayvertex,label={[label distance=-0.05cm]120:{$\color{gray}{v_0}$}}] (v0) at (0.9, 1.5) {};
\node[redvertex,label={[label distance=-0.05cm]-120:{$\color{red}{q^2}$}}] (q2) at (-1.2, 0) {};

\begin{scope}[grayedge]
\draw [dashed] plot [smooth, tension = 1] coordinates {(0.9,1.5) (0,0.85) (-1.3,1) (-1.2,0) (-0.2,0.2) (0.2,-0.1)};

\draw [dashed] plot [smooth, tension = 1] coordinates {(1,0) (1, 0.4) (1.15,1.1) (0.9,1.5)};
\end{scope}

\begin{scope}[greenedge]
\draw (v2) to[in=50, out = 130] (q2);
\draw (v2) to[in=-50, out = -130] (q2);
\end{scope}

\end{tikzpicture} \hspace{2cm}						\begin{tikzpicture}
\tikzset{bluevertex/.style={shape=circle, draw=blue, scale=0.3, fill=blue}}
\tikzset{grayvertex/.style={shape=circle, draw=gray, scale=0.3, fill=gray}}
\tikzset{redvertex/.style={shape=circle, draw=red, scale=0.3, fill=red}}
\tikzset{blackvertex/.style={shape=circle, draw=gray, scale=0.3, fill=black}}
\tikzset{blueedge/.style={shape=circle, draw=blue, thick}}
\tikzset{grayedge/.style={shape=circle, draw=gray, thick}}
\tikzset{greenedge/.style={shape=circle, draw=green!70!black, thick}}

\tikzset{noNode/.style={draw=none}}



\node[grayvertex] (v1) at (0,0) {};
\node[blackvertex] (v2') at (0,-2) {};
\node[blackvertex] (v2) at (0,2) {};
\node[grayvertex] (v0) at ({2*tan(30)}, 0) {};
\node[grayvertex] (v0') at ({-2*tan(30)}, 0) {};
\node[redvertex] (q2) at (0.6, 0) {};
\node[redvertex] (q2') at (-0.6, 0) {};


\begin{scope}[grayedge, dashed]
\draw   (v0') -- (v2')--(v0);
\draw (v0) -- (v1) -- (v0') ;
\draw  (v2) --(v0);
\draw (v0') -- (v2);
\end{scope}

\begin{scope}[greenedge]
\draw (q2) -- (v2) -- (q2') -- (v2') -- (q2);
\end{scope}

\end{tikzpicture}
			\end{center}
			\caption{On the left a schematic representation of a lopsp-operation in the proof of Lemma~\ref{lem:type2_cutpath}} is shown. The dashed lines show the cut-path $P$. Only relevant vertices and edges are drawn. On the right the $P$-diamond of the operation is shown.
			\label{fig:type2_cutpath_proof}
		\end{figure}
		
		\item[\fbox{$v_1\in Q$:}] In this case $v_1$ is of type 0, if it would be of type 1 it would have degree four in $O$ which is impossible by Definition~\ref{def:lopsp}. Let $Q'$ be the graph $Q\setminus \{v_0,v_2\}$. As before, this is an induced graph by Lemma~\ref{lem:path_is_induced}. Applying Lemma~\ref{lem:remove_induced_subgraph} gives a path $P_2$ between $v_0$ and $v_2$ that does not contain any vertices of $Q'$. Now $P_2$ and the subpath of $\widetilde{Q}$ between $v_0$ and $v_1$ form a cut-path together that satisfies the properties of (b). The path $R$ is the subpath of ${Q}$ between $v_1$ and $q_2$. 
	\end{itemize}
\end{proof}

Lemma~\ref{lem:polyhedral_cut} will be used several times in the proofs of Lemma~\ref{lem:diff_v0sets} and Lemma~\ref{lem:no_2_in_v0set}. 
The idea of those proofs is that we assume that $O(G)$ has a 2-cut, and prove that certain components are `small' enough so that they would also appear when applying $O$ to a cube (or other polyhedral embeddings). With Lemma~\ref{lem:polyhedral_cut} that leads to a contradiction. In this article Lemma~\ref{lem:polyhedral_cut} is only used with the cube as $G'$, but it is formulated for general polyhedral embeddings.

\begin{lemma}\label{lem:polyhedral_cut}
	Let $G$ be an embedded graph, and let $O$ be a $c3$-lopsp-operation. Let $H$ be a subgraph of $B_{O(G)}$. There is no face $f$ in $H$ and two 0-vertices $x$ and $y$ in $O(G)$ with all of the following properties:
	\begin{enumerate}[(1)]
		\item $f$ is simple and there is a polyhedral embedding $G'$ such that there is a simple face $f'$ in a subgraph $H'$ of $B_{O(G')}$ with $IC(f)\cong IC(f')$.
		\item For every vertex of type 1 in the boundary of $f$ its two neighbours of type 0 are also in the boundary of $f$.
		\item There is a 0-vertex $a$ in the interior of $f$ such that every type-2 path from $a$ to a 0-vertex in the boundary of $f$ contains $x$ or $y$.
		\item $x$ and $y$ each appear at most once in the facial walk of $f$.
	\end{enumerate} 
\end{lemma}
\begin{proof}
	Assume that such a face $f$ exists. As $IC(f)\cong IC(f')$ and $a$ is in the interior of $f$, there is a unique vertex $a'$ in $B_{O(G')}$ that corresponds to $a$. Because of $(4)$, $x$ and $y$ each correspond to at most one vertex of $IC(f)$. Therefore, if $x$ and $y$ are in the boundary or interior of $f$ they correspond to unique vertices $x'$ and $y'$ of $B_{O(G')}$ in $f'$. If w.l.o.g.\ $x$ is not then as $f$ is simple there is no type-2 path from $a$ to $x$ without a vertex in the boundary of $f$. 
	
	Condition $(2)$ implies that if a non-trivial type-2 path contains a 1-vertex in the boundary of $f$ then it also contains a 0-vertex in the boundary of $f$. As $f'$ is simple, any type-2 path from $a'$ to a 0-vertex outside of $f'$ contains a 0-vertex in the boundary of $f'$. By $(3)$ and the fact that $IC(f)\cong IC(f')$, such a path also contains $x'$ or $y'$. It follows that $\{x',y'\}$ is a cut of $O(G')$ of size at most two and therefore $O(G')$ is not 3-connected. As $G'$ is a polyhedral embedding and $O$ is $c3$, it follows by Theorem 2 in \cite{brinkmann2021local} that $O(G')$ is also a polyhedral embedding and therefore it is 3-connected, a contradiction. 
\end{proof}

\begin{lemma}\label{lem:diff_v0sets}
	Let $G$ be a simple 3-connected embedded graph, let $O$ be a $c3$-lopsp-operation and let $\{x,y\}$ be a set of vertices in $O(G)$. If $O$ is edge-preserving or $G^*$ is simple, then every component $A$ of $O(G)\setminus \{x,y\}$ contains a vertex in a vertex-shadow.
\end{lemma}
\begin{proof}
	
Assume that $O$ is edge-preserving or $G^*$ is simple, and that there is a component $A$ of $O(G)\setminus \{x,y\}$ that does not contain a vertex in a vertex-shadow. Let $P$ be a cut-path as in Lemma \ref{lem:type2_cutpath}. We discuss the two different cases separately. In \textbf{case (a)} $P$ satisfies (a) of Lemma~\ref{lem:type2_cutpath}, in \textbf{case (b)} $P$ satisfies (b) of Lemma~\ref{lem:type2_cutpath}.

\begin{itemize}
	\item[\textbf{case (a)}] Let $Q$ be the subpath of $P_{v_0,v_2}=:\widetilde{Q}$ between the vertex $q_0$ in $N_0(v_0)$ and the vertex $q_2$ in $N_0(v_2)$. For every double chamber $D$, let $\widetilde{Q}_L^D$ denote the left copy of $\widetilde{Q}$ in that double chamber. We will also use a superscript $D$ to denote the copies of $q_0$, $q_2$ and $Q$ in $\widetilde{Q}_L^D$. The vertex $q_0^D$ is not in component $A$ by assumption. Consider the double chambers $D_L$ and $D_R$ that share a 1-side with double chamber $D$. The only vertices of $Q^{D_L}$ and $Q^{D_R}$ that are in a face-shadow are $q_2^{D_L}$ and $q_2^{D_R}$. The path $Q^D$ and the paths $Q^{D_L}$ and $Q^{D_R}$ extended with paths in the face-shadow from $q_2^{D_L}$ and $q_2^{D_R}$ to $q^D_2$ form three paths from different vertex-shadows to $q_2^D$. The paths are vertex-disjoint because $D$, $D_L$ and $D_R$ have different 2-sides. As $x$ and $y$ can only be in two of these paths, $q_2^D$ is not in component $A$. It follows that no copy of $q_2$ is in component $A$.
	
	Let $\widetilde{Q}_G$ be the embedded subgraph $\pi^{-1}(\widetilde{Q})$ of $B_{O(G)}$. Note that every face of $\widetilde{Q}_G$ is simple. Let $a$ be a vertex in component $A$. If there is a vertex of $A$ in $\widetilde{Q}_G$ then let $a$ be such a vertex. Then it appears in a copy of $Q$ between $x$ and $y$, for the endpoints of copies of $Q$ are not in $A$. Remove this path from $\widetilde{Q}_G$. Now $a$ is in the interior of a simple face $f$ of the modified $\widetilde{Q}_G$. The two possible configurations of this face are shown in Figure~\ref{fig:case1}. There are two cases:
	
	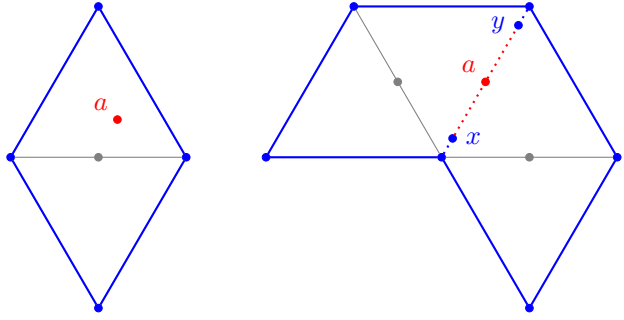
\begin{figure}
		\centering
		\begin{tikzpicture}
\tikzset{bluevertex/.style={shape=circle, draw=blue, scale=0.3, fill=blue}}
\tikzset{grayvertex/.style={shape=circle, draw=gray, scale=0.3, fill=gray}}
\tikzset{redvertex/.style={shape=circle, draw=red, scale=0.3, fill=red}}
\tikzset{blueedge/.style={shape=circle, draw=blue, thick}}
\tikzset{grayedge/.style={shape=circle, draw=gray}}

\tikzset{noNode/.style={draw=none}}


\node[draw=none] (_) at (0, {-4.82*tan(30)}) {};

\node[grayvertex] (v1) at (0,0) {};
\node[bluevertex] (v2') at (0,-2) {};
\node[bluevertex] (v2) at (0,2) {};
\node[bluevertex] (v0) at ({2*tan(30)}, 0) {};
\node[bluevertex] (v0') at ({-2*tan(30)}, 0) {};

\node[redvertex,label={[label distance=-0.05cm]120:{$\color{red}{a}$}}] (b) at (0.25, 0.5) {};

\begin{scope}[grayedge]
\draw  (v2) --(v0) -- (v2') --(v0') -- (v2);
\draw (v0) -- (v1) -- (v0');
\end{scope}

\begin{scope}[blueedge]
\draw  (v2) --(v0) -- (v2') --(v0') -- (v2);
\end{scope}

\end{tikzpicture} \qquad
		\begin{tikzpicture}
\tikzset{bluevertex/.style={shape=circle, draw=blue, scale=0.3, fill=blue}}
\tikzset{grayvertex/.style={shape=circle, draw=gray, scale=0.3, fill=gray}}
\tikzset{redvertex/.style={shape=circle, draw=red, scale=0.3, fill=red}}
\tikzset{blueedge/.style={shape=circle, draw=blue, thick}}
\tikzset{grayedge/.style={shape=circle, draw=gray}}

\tikzset{noNode/.style={draw=none}}


\node[draw=none] (_) at (0, {-4.82*tan(30)}) {};

\node[grayvertex] (v1) at (0,0) {};
\node[bluevertex] (v2') at (0,-2) {};
\node[bluevertex] (v2) at (0,2) {};
\node[bluevertex] (v0) at ({2*tan(30)}, 0) {};
\node[bluevertex] (v0') at ({-2*tan(30)}, 0) {};
\node[grayvertex] (v1l) at ({-3*tan(30)}, 1) {};

\node[redvertex,label={[label distance=-0.05cm]120:{$\color{red}{a}$}}] (b) at ({-1*tan(30)}, 1) {};
\node[bluevertex] (v0l) at ({-4*tan(30)}, 2) {};
\node[bluevertex] (v2l) at ({-6*tan(30)}, 0) {};

\node[bluevertex,label={[label distance=-0.0cm]180:{$\color{blue}{y}$}}] (y) at ({-0.25*tan(30)}, 1.75) {};
\node[bluevertex,label={[label distance=-0.0cm]0:{$\color{blue}{x}$}}] (x) at ({-1.75*tan(30)}, 0.25) {};

\begin{scope}[grayedge]
\draw (v0) -- (v1) -- (v0') --(v1l) -- (v0l);
\end{scope}

\begin{scope}[blueedge]
\draw  (v2) --(v0) -- (v2') --(v0') -- (v2l) -- (v0l) -- (v2);
\draw[dotted] (v0') --(x);
\draw[dotted] (v2) --(y);
\end{scope}

\draw[red, thick, dotted] (x) -- (b) -- (y);

\end{tikzpicture}
		\caption{This figure shows the two possible configurations in the proof of Lemma~\ref{lem:diff_v0sets}, \textbf{case (a)}. The blue lines show $f$, the gray lines show $D_G$. In the right image the vertices $x$ and $y$ are not drawn as vertices of $f$, but they could be the 0- or 2-point of the 1-side containing $a$.}
		\label{fig:case1}
	\end{figure}
	
	\begin{itemize}
		\item  \textbf{$G^*$ is simple or neither $x$ nor $y$ are a 2-point.} We prove that $f$, $x$ and $y$ satisfy the properties of Lemma~\ref{lem:polyhedral_cut}, which is a contradiction with that lemma.
		\begin{enumerate}[(1)]
			\item Take for example the cube as $G'$.
			\item By definition of $\widetilde{Q}_G$.
			\item It suffices to prove that no 0-vertex in the boundary of $f$ is in component $A$. If we did not remove a copy of $\widetilde{Q}$ this follows by the choice of $a$. Otherwise the only way that a 0-vertex in the boundary of $f$ could be in component $A$ is if $x$ and $y$ are in another copy $\widetilde{Q}$ together, which is only possible if one of them is a 2-point and $G^*$ has a loop. We assumed that that is not the case.		 
			\item The only way that $f$ is not a cycle is if the 2-points drawn in Figure~\ref{fig:case1} do not represent different 2-points, i.e. if $G^*$ is not simple. In that case neither $x$ nor $y$ are a 2-point so they appear only once in the facial walk of $f$.
		\end{enumerate}

		\item \textbf{$G^*$ is not simple, and therefore $O$ is edge-preserving, and $x$ or $y$ is a 2-point of $D_G$.} For every diamond around an edge of $G$, add an edge-path (Definition~\ref{def:edge-path}) that does not contain either of the 2-points of the diamond to $\widetilde{Q}_G$. Such an edge-path exists because of Lemma~\ref{lem:edge-path}. Let $\widetilde{Q}^+_G$ be the result. 
		
		Let $f$ be the face of $\widetilde{Q}^+_G$ that has $a$ in its interior. The different cases are shown in Figure~\ref{fig:case1+}. Note that $a$ is not in one of the added edge-paths because $x$ or $y$ is a 2-point of $D_G$. Again we prove that $f$, $x$ and $y$ satisfy the properties in Lemma~\ref{lem:polyhedral_cut}, which is a contradiction.
		\begin{enumerate}[(1)]
			\item Take for example the cube as $G'$.
			\item By definition of $\widetilde{Q}_G^+$
			\item 
			Again it follows from the choice of $a$ and the restrictions on the location of $x$ and $y$ that no 0-vertex in the boundary of $f$ is in component $A$. 	
			
			\item In both cases the boundary of $f$ is a cycle.
		\end{enumerate}
	\end{itemize} 

	\begin{figure}
	\centering
	\begin{tikzpicture}
\tikzset{bluevertex/.style={shape=circle, draw=blue, scale=0.3, fill=blue}}
\tikzset{grayvertex/.style={shape=circle, draw=gray, scale=0.3, fill=gray}}
\tikzset{redvertex/.style={shape=circle, draw=red, scale=0.3, fill=red}}
\tikzset{blueedge/.style={shape=circle, draw=blue, thick}}
\tikzset{grayedge/.style={shape=circle, draw=gray}}

\tikzset{noNode/.style={draw=none}}


\node[draw=none] (_) at (0, {-4.82*tan(30)}) {};

\node[grayvertex] (v1) at (0,0) {};
\node[grayvertex] (v2') at (0,-2) {};
\node[bluevertex] (v2) at (0,2) {};
\node[bluevertex] (v0) at ({2*tan(30)}, 0) {};
\node[bluevertex] (v0') at ({-2*tan(30)}, 0) {};

\node[redvertex,label={[label distance=-0.05cm]120:{$\color{red}{a}$}}] (b) at (0.3, 1) {};



\begin{scope}[grayedge]
\draw   (v0') -- (v2')--(v0);
\draw (v0) -- (v1) -- (v0') ;
\end{scope}

\begin{scope}[blueedge]
\draw  (v2) --(v0);
\draw (v0') -- (v2);
\end{scope}


\draw [blueedge] plot [smooth] coordinates {({-2*tan(30)},0) (0,-0.7) ({1*tan(30)},0.2) ({2*tan(30)},0)};

\end{tikzpicture}\qquad\qquad
	\begin{tikzpicture}
\tikzset{bluevertex/.style={shape=circle, draw=blue, scale=0.3, fill=blue}}
\tikzset{grayvertex/.style={shape=circle, draw=gray, scale=0.3, fill=gray}}
\tikzset{redvertex/.style={shape=circle, draw=red, scale=0.3, fill=red}}
\tikzset{blueedge/.style={shape=circle, draw=blue, thick}}
\tikzset{grayedge/.style={shape=circle, draw=gray}}

\tikzset{noNode/.style={draw=none}}


\node[draw=none] (_) at (0, {-4.82*tan(30)}) {};

\node[grayvertex] (v1) at (0,0) {};
\node[grayvertex] (v2') at (0,-2) {};
\node[bluevertex,label={[label distance=-0.0cm]0:{$\color{blue}{y}$}}] (y) at (0,2) {};
\node[bluevertex] (v0) at ({2*tan(30)}, 0) {};
\node[bluevertex] (v0') at ({-2*tan(30)}, 0) {};
\node[grayvertex] (v1l) at ({-3*tan(30)}, 1) {};

\node[redvertex,label={[label distance=-0.05cm]120:{$\color{red}{a}$}}] (b) at ({-1*tan(30)}, 1) {};
\node[bluevertex] (v0l) at ({-4*tan(30)}, 2) {};
\node[grayvertex] (v2l) at ({-6*tan(30)}, 0) {};


\node[bluevertex,label={[label distance=-0.0cm]0:{$\color{blue}{x}$}}] (x) at ({-1.75*tan(30)}, 0.25) {};

\begin{scope}[grayedge]
\draw  (v0l) -- (v2l) -- (v0') -- (v2')--(v0);
\draw (v0) -- (v1) -- (v0') --(v1l) -- (v0l);
\end{scope}

\begin{scope}[blueedge]
\draw  (y) --(v0);
\draw (v0l) -- (y);
\draw[dotted] (v0') --(x);
\end{scope}

\draw[red, thick, dotted] (x) -- (b) -- (y);

\draw [blueedge] plot [smooth] coordinates {({-4*tan(30)},2) ({-3*tan(30) - 0.7},{1-0.7*tan(30)}) ({-2.5*tan(30) +0.2},{0.5+0.2*tan(30)}) ({-2*tan(30)},0)};
\draw [blueedge] plot [smooth] coordinates {({-2*tan(30)},0) (0,-0.7) ({1*tan(30)},0.2) ({2*tan(30)},0)};

\end{tikzpicture}
	\caption{This figure shows the two possible configurations in the proof of Lemma~\ref{lem:diff_v0sets}, \textbf{case (a)}. The blue lines show $f$, the gray lines show $D_G$. In the right image $x$ is not drawn as a vertex of $f$, but it could be the 0-point of the 1-side containing $a$.}
	\label{fig:case1+}
	\end{figure}
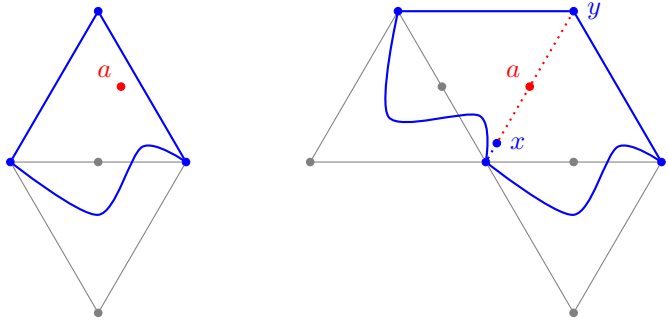

	 \item[\textbf{case (b)}] 
	 Let $\widetilde{Q}$ be the union of $P_{v_0,v_1}$ and the path $R$ from Lemma~\ref{lem:type2_cutpath}. If $v_2$ is of type 2 also add $v_2$ and its edge to the vertex $q_2\in N_0(v_2)$ in $R$. If $v_1$ is of type 0 then $\widetilde{Q}$ is a path, if $v_1$ is of type 1 then $\widetilde{Q}$ is the union of a path and an edge between $v_1$ and the vertex $q_1$ in the path. Let $Q$ be the type-2 path in $\widetilde{Q}$ between $q_0\in N_0(v_0)$ and $q_2$. In $O_P$ there is at least one path corresponding to the path $Q$. If there are two, choose one of them. There is a copy of that path in every double chamber. Let $Q^D$ denote the copy in double chamber $D$. Let $q_0^D$, $q_1^D$ and $q_2^D$ be the copies of $q_0$, $q_1$ and $q_2$ in $Q^D$.
	 
	 Except for perhaps 0- or 2-points, the paths $Q^D$ are disjoint for different $D$. Each path $Q^D$ contains only one vertex in a face-shadow, namely $q_2^D$. It follows as in the first paragraph of \textbf{case (a)} that we can construct three vertex-disjoint paths from vertices in vertex-images to any copy of $q_2$, which proves that no copy of $q_2$ is in component $A$. 
	 
	 In every double chamber there are two vertex-disjoint type-2 paths on the 2-side from the copies of $q_0$ to $q_1$. The type-2 path $R$ described in Lemma~\ref{lem:type2_cutpath} induces a path from the copy of $q_2$ to the copy of $q_1$. These three paths are vertex-disjoint paths from vertices in component $A$ to $q_1$. As there are three such paths and only two vertices $x$ and $y$, it follows that no copy of $q_1$ is in component $A$.
	 
	 Let $\widetilde{Q}_G$ be the subgraph $\pi^{-1}(\widetilde{Q})$ of $B_{O(G)}$. Let $a$ be a vertex in $O(G)\setminus \{x,y\}$ that is in component $A$. If possible choose $a$ in $\widetilde{Q}_G$. If $a$ is in $\widetilde{Q}_G$ then $x$ and $y$ are in the same subpath of $\widetilde{Q}_G$ as $a$ between copies of $q_i$ and $q_j$. Remove this path from $\widetilde{Q}_G$. Then $a$ is in the interior of a face $f$ of the modified $\widetilde{Q}_G$. As before, we prove that $f$, $x$ and $y$ satisfy the properties of Lemma~\ref{lem:polyhedral_cut}, which is a contradiction with that Lemma.
	 \begin{enumerate}[(1)]
	 	\item Take for example the cube as $G'$.
	 	\item By definition of $\widetilde{Q}_G$.
	 	\item Again it follows from the choice of $a$ and the restrictions on the location of $x$ and $y$ that no 0-vertex in the boundary of $f$ is in component $A$. 
	 	
	 	\item In the first two cases of Figure~\ref{fig:case2} the boundary of $f$ is a cycle. In the third case the two drawn 2-points may be the same, but $x$ and $y$ cannot be this 2-point as they are both on a 2-side. It follows that $x$ and $y$ each appear at most once in the facial walk of $f$.
	 \end{enumerate}

	 	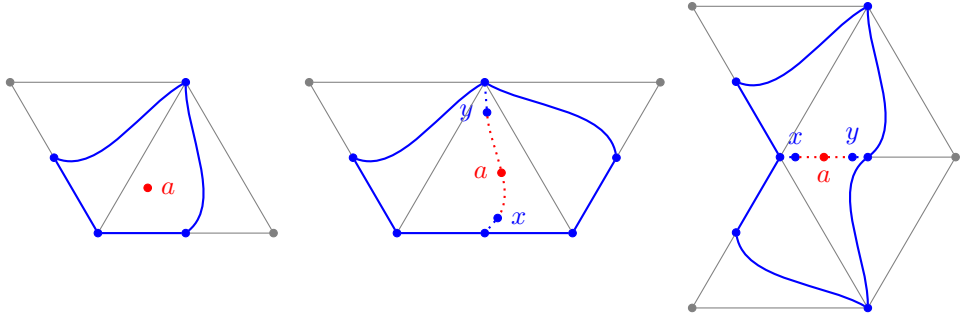
\begin{figure}
	 	\centering
	 	\begin{tikzpicture}[baseline=-31pt]
\tikzset{bluevertex/.style={shape=circle, draw=blue, scale=0.3, fill=blue}}
\tikzset{grayvertex/.style={shape=circle, draw=gray, scale=0.3, fill=gray}}
\tikzset{redvertex/.style={shape=circle, draw=red, scale=0.3, fill=red}}
\tikzset{blueedge/.style={shape=circle, draw=blue, thick}}
\tikzset{grayedge/.style={shape=circle, draw=gray}}

\tikzset{noNode/.style={draw=none}}

\clip (-2.4,-0.1) rectangle (1.3,2.2);

\node[draw=none] (_) at (0, {-4.82*tan(30)}) {};

\node[bluevertex] (v1) at (0,0) {};
\node[bluevertex] (v2) at (0,2) {};
\node[grayvertex] (v0) at ({2*tan(30)}, 0) {};
\node[bluevertex] (v0') at ({-2*tan(30)}, 0) {};
\node[bluevertex] (v1l) at ({-3*tan(30)}, 1) {};

\node[redvertex,label={[label distance=-0.0cm]-0:{$\color{red}{a}$}}] (b) at (-0.5, 0.6) {};
\node[grayvertex] (v0l) at ({-4*tan(30)}, 2) {};


\begin{scope}[grayedge]
\draw (v0) -- (v1);
\draw (v1l) -- (v0l);
\draw  (v0') -- (v2) --(v0) ;
\draw (v2) -- (v0l);
\end{scope}

\begin{scope}[blueedge]
\draw (v1) --(v0') -- (v1l);
\draw (v2) edge[looseness=0.8, in= 40, out =-90] (v1);
\draw (v2) edge[looseness=0.8, in= -20, out =-150] (v1l);
\end{scope}

\end{tikzpicture}
	 	\begin{tikzpicture}[baseline=-31pt]
\tikzset{bluevertex/.style={shape=circle, draw=blue, scale=0.3, fill=blue}}
\tikzset{grayvertex/.style={shape=circle, draw=gray, scale=0.3, fill=gray}}
\tikzset{redvertex/.style={shape=circle, draw=red, scale=0.3, fill=red}}
\tikzset{blueedge/.style={shape=circle, draw=blue, thick}}
\tikzset{grayedge/.style={shape=circle, draw=gray}}

\tikzset{noNode/.style={draw=none}}

\clip (-2.4,-0.1) rectangle (2.4,2.2);

\node[draw=none] (_) at (0, {-4.82*tan(30)}) {};

\node[bluevertex] (v1) at (0,0) {};
\node[bluevertex] (v2) at (0,2) {};
\node[bluevertex] (v0) at ({2*tan(30)}, 0) {};
\node[bluevertex] (v0') at ({-2*tan(30)}, 0) {};
\node[bluevertex] (v1l) at ({-3*tan(30)}, 1) {};

\node[redvertex,label={[label distance=-0.0cm]-180:{$\color{red}{a}$}}] (b) at (0.22, 0.8) {};
\node[grayvertex] (v0r) at ({4*tan(30)}, 2) {};
\node[grayvertex] (v0l) at ({-4*tan(30)}, 2) {};
\node[bluevertex] (v1r) at ({3*tan(30)}, 1) {};

\node[bluevertex,label={[label distance=-0.0cm]180:{$\color{blue}{y}$}}] (y) at (0.03, 1.6) {};
\node[bluevertex,label={[label distance=-0.0cm]0:{$\color{blue}{x}$}}] (x) at (0.17, 0.2) {};

\begin{scope}[grayedge]
\draw (v1l) -- (v0l);
\draw  (v0') -- (v2) --(v0) ;
\draw (v1r) --(v0r) --(v2) -- (v0l);
\end{scope}

\begin{scope}[blueedge]
\draw (v1r) -- (v0) --(v1) --(v0') -- (v1l);
\draw (v2) edge[looseness=0.8, in= 100, out =-30] (v1r);
\draw (v2) edge[looseness=0.8, in= -20, out =-150] (v1l);
\draw[dotted] (x) -- (v1)
			    (y) -- (v2);
\end{scope}

\draw[red, thick, dotted] (y) edge[looseness=0.8, in= 60, out =-85] (x);
\end{tikzpicture}
	 	\begin{tikzpicture}
\tikzset{bluevertex/.style={shape=circle, draw=blue, scale=0.3, fill=blue}}
\tikzset{grayvertex/.style={shape=circle, draw=gray, scale=0.3, fill=gray}}
\tikzset{redvertex/.style={shape=circle, draw=red, scale=0.3, fill=red}}
\tikzset{blueedge/.style={shape=circle, draw=blue, thick}}
\tikzset{grayedge/.style={shape=circle, draw=gray}}

\tikzset{noNode/.style={draw=none}}

\clip (-2.4,-2.1) rectangle (1.3,2.2);

\node[draw=none] (_) at (0, {-4.82*tan(30)}) {};

\node[bluevertex] (v1) at (0,0) {};
\node[bluevertex] (v2') at (0,-2) {};
\node[bluevertex] (v2) at (0,2) {};
\node[grayvertex] (v0) at ({2*tan(30)}, 0) {};
\node[bluevertex] (v0') at ({-2*tan(30)}, 0) {};
\node[bluevertex] (v1l) at ({-3*tan(30)}, 1) {};
\node[bluevertex] (v1l') at ({-3*tan(30)}, -1) {};

\node[redvertex,label={[label distance=-0.0cm]-90:{$\color{red}{a}$}}] (b) at ({-1*tan(30)}, 0) {};
\node[grayvertex] (v0l) at ({-4*tan(30)}, 2) {};
\node[grayvertex] (v0l') at ({-4*tan(30)}, -2) {};

\node[bluevertex,label={[label distance=-0.0cm]90:{$\color{blue}{y}$}}] (y) at ({-0.35*tan(30)}, 0) {};
\node[bluevertex,label={[label distance=-0.0cm]90:{$\color{blue}{x}$}}] (x) at ({-1.65*tan(30)}, 0) {};

\begin{scope}[grayedge]
\draw (v0) -- (v1);
\draw (v1l) -- (v0l);
\draw  (v0') -- (v2) --(v0) -- (v2') --(v0');
\draw (v1l') -- (v0l') -- (v2');
\draw (v2) -- (v0l);
\end{scope}

\begin{scope}[blueedge]
\draw (v0') -- (v1l);
\draw (v2) edge[looseness=0.8, in= 40, out =-90] (v1);
\draw (v2') edge[looseness=0.8, in= 220, out =90] (v1);
\draw (v2') edge[looseness=0.8, in= -80, out =150] (v1l');
\draw (v2) edge[looseness=0.8, in= -20, out =-150] (v1l);
\draw  (v1l') -- (v0');
\draw[dotted] (x) -- (v0')
			    (y) -- (v1);
\end{scope}

\draw[red, thick, dotted] (x) -- (b)--(y);
\end{tikzpicture}
	 	\caption{This figure shows the three possible configurations in the proof of Lemma~\ref{lem:diff_v0sets}, \textbf{case (b)}. The blue lines show $f$, the gray lines show $D_G$. The vertices $x$ and $y$ are not drawn as vertices of $f$, but either one could be a point of $D_G$ in the path containing $a$.}
	 	\label{fig:case2}
	 \end{figure}
\end{itemize}
\end{proof}

\begin{lemma}\label{lem:no_2_in_v0set}
	Let $G$ be a simple 3-connected embedded graph, let $O$ be a $c3$-lopsp-operation different from Dual and let $\{x,y\}$ be a set of vertices in $O(G)$. If $O$ is edge-preserving or $G^*$ is simple, then in every component $A$ of $O(G)\setminus \{x,y\}$ there is a vertex that is in a vertex-shadow of a vertex of $G$ that is not broken by $\{x,y\}$.
\end{lemma}
\begin{proof}
	If $t(v_0)=0$ this follows immediately from Lemma~\ref{lem:diff_v0sets}, so assume that $t(v_0)=2$. If $\{x,y\}$ breaks no vertices of $G$ we are done, so assume that a vertex $v$ is broken by $\{x,y\}$. It follows from Lemma~\ref{lem:verteximage_cycle} that $v$ is the only vertex broken by $\{x,y\}$ and $x$ and $y$ are in $S_O(v)$. By Lemma~\ref{lem:diff_v0sets}, every component contains at least one vertex in a vertex-shadow. Assume that there is a component $A$ of $O(G)\setminus \{x,y\}$ that does not contain vertices of other vertex-shadows than $S_O(v)$. 
	
	Let the cut-path $P$ of $O$ be as in Lemma~\ref{lem:type2_cutpath}. We define a type-2 subgraph $Q$ in $O$ and prove that no vertex of $O(G)$ in $\pi^{-1}(Q)$ is in component $A$ for the two different cases of cut-paths in Lemma~\ref{lem:type2_cutpath} separately:
		
	\begin{itemize}
		\item[\textbf{case (a):}]
		Let $Q$ be the type-2 subpath of $P_{v_0,v_2}=:\widetilde{Q}$ between $q_0\in S_O(v_0)$ and $q_2\in S_O(v_2)$. Assume that there is a vertex $a$ in component $A$ that is in a copy $Q^{a}$ of $Q$. Let $q_2^{a}$ be the copy of $q_2$ in $Q^{a}$. 
		
		The vertices $x$ and $y$ are in $S_O(v)$ so they are not between $q_2^{a}$ and $a$ in $Q^{a}$. This implies that $q_2^{a}$ and $a$ are in the same component of $O(G) \setminus \{x,y\}$ and thus $q_2^{a}$ is in component $A$. Let $f$ be the face of $G$ with $q_2^{a}\in S_O(f)$. In the two double chambers that share the 1-side containing $Q^{a}$ there are two other copies of $Q$. Adding the vertices of $S_O(f)$ in those double chambers to those copies of $Q$ we get vertex-disjoint paths from two vertices in different vertex-shadows than $S_O(v)$ to $q_2^{a}$. For $q_2^{a}$ to be in $A$, those two paths must each contain one vertex of $\{x,y\}$. As $x$ and $y$ cannot be in the copies of $Q$ in those paths, they must both be in $S_O(f)$. Now $v,x,f,y$ is a 4-cycle in two adjacent copies of $O_P$ that has at least one vertex, $q_2^{a}$, of type 0 on its inside. By Lemma~\ref{lem:c3char} $O$ is not $c3$, a contradiction. We have proved that there are no vertices of component $A$ in copies of $Q$.

		\item[\textbf{case (b):}] 
		Let $\widetilde{Q}$ be the union of $P_{v_0,v_1}$ and the path $R$ from Lemma~\ref{lem:type2_cutpath}. If $v_2$ is of type 2 also add $v_2$ and its edge to the vertex $q_2\in N_0(v_2)$ in $R$. 
		Consider the type-2 path on a 2-side of a double chamber between two vertices in different vertex-shadows. As $x$ and $y$ are in the same vertex-shadow and those vertices cannot be in the path except as one endpoint, no vertex of the path is in component $A$. It follows that no vertex in a copy of $P_{v_0,v_1}$ is in component $A$. 
		
		Now assume that there is a vertex $a\in A$ in a copy $R^{a}$ of $R$. Let $D^{a}$ be the double chamber containing $R^{a}$ and let $D^{a}_L$ and $D^{a}_R$ be the two double chambers sharing a 1-side with $D^{a}$. There are three paths from copies of $q_1$ to the vertex $q_2^{a}$ in $R^{a}$, each consisting of a copy of $R$ (in $D^{a}_L$, $D^{a}$ or $D^{a}_R$) and vertices in the face-shadow containing $q_2^{a}$. As the only vertices in a face-shadow that are in a copy of $R$ are the copies of $q_2$, these three paths are vertex-disjoint. This implies that there is at least one path without $x$ or $y$ from a vertex that is not in $A$ to $q_2^{a}$ and therefore also to $a$, a contradiction.
	\end{itemize}

	In both cases we can now consider the graph $\widetilde{Q}_G:=\pi^{-1}(\widetilde{Q})$. Let $f$ be a face of $\widetilde{Q}_G$ that has a vertex of $A$ in its interior. We get a contradiction by proving that $f$, $x$ and $y$ satisfy the properties in Lemma~\ref{lem:polyhedral_cut}.
	
	\begin{enumerate}[(1)]
		\item Take for example the cube as $G'$.
		\item By definition of $\widetilde{Q}_G$.
		\item We proved this for each case.
		\item In case (b) the face is a cycle. In case (a) the face may contain only one distinct 2-point instead of two. However, $x$ and $y$ cannot be this 2-point as they are in a vertex-shadow and $O$ is not Dual. It follows that $x$ and $y$ each appear at most once in the facial walk of $f$.
	\end{enumerate}
\end{proof}

\section{Main theorems}\label{sec:main}

We can now put together the lemmas from the previous sections to prove that edge-preserving operations preserve 3-connectivity of simple embedded graphs.

\begin{theorem}\label{thm:main_1}
	Let $G$ be a simple 3-connected graph, and let $O$ be a $c3$-lopsp-operation different from Dual. If $O$ is edge-preserving or $G^*$ is simple 
	then $O(G)$ is 3-connected.
\end{theorem}
\begin{proof}
	Assume that $G$ is 3-connected, that $O$ is edge-preserving or $G^*$ is simple and that $O(G)$ is not 3-connected. It follows from Lemma~\ref{lem:size_OG} that $O(G)$ has at least four vertices so $O(G)$ has a 2-cut $X$.
	
	By Lemma~\ref{lem:no_2_in_v0set} there exist two vertices $a$ and $b$ in different components of $O(G)\setminus X$ that are also in different vertex-shadows $S_O(v)$ and $S_O(w)$ such that $X$ does not break $v$ or $w$.
	
	Let $G'$ be the embedded graph obtained from $G$ by removing all the edges and vertices that are broken by $X$. By Corollary~\ref{cor:breaking_vertices} and Lemma~\ref{lem:2_broken} it follows that the sum of the removed vertices and removed edges not adjacent to those vertices is at most 2. As $G$ is 3-connected, this implies that $G'$ is connected. Let $P$ be a path in $G'$ from $v$ to $w$.
	
	It now follows from the definitions of breaking vertices and edges that $P$ induces a path from $a$ to $b$ in $O(G)\setminus X$. This is impossible as $a$ and $b$ were in different components. It follows that $O(G)$ is 3-connected.
\end{proof}

This theorem immediately implies that edge-preserving lopsp-operations always preserve 3-connectivity, but it does not prove that this is not true for edge-breaking operations. In the following theorem we give an example of a 3-connected embedded graph to which applying any edge-breaking lopsp-operation results in a graph that is not 3-connected.

\begin{theorem}\label{thm:edge-breaking_breaks}
	For every edge-breaking $c3$-operation $O$, there exist simple 3-connected embedded graphs $G$ of arbitrary genus $>0$ such that the (multi)graph $O(G)$ is not 3-connected.
\end{theorem}
\begin{proof}
	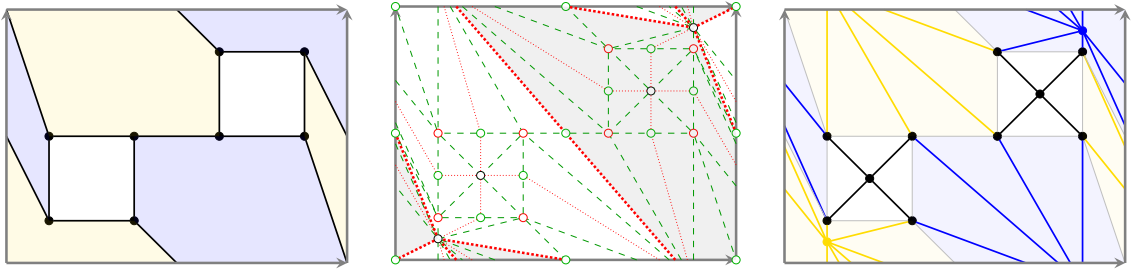
\begin{figure}			
		\centering
		\scalebox{0.8}{\begin{tikzpicture}[scale=0.7]
\tikzset{normal/.style={shape=circle, draw=black, scale=0.4,fill=black}}

\node[normal] (v1) at (-3,0) {};
\node[normal] (v2) at (-1,0) {};
\node[normal] (v3) at (-3,-2) {};
\node[normal] (v4) at (-1,-2) {};

\node[normal] (v1') at (3,0) {};
\node[normal] (v2') at (1,0) {};
\node[normal] (v3') at (3,2) {};
\node[normal] (v4') at (1,2) {};

\begin{scope}[blue,opacity = 0.1]
\fill (3,0) -- (4,-3) -- (0,-3) -- (-1,-2) -- (-1,0) -- cycle;
\fill (-4,3) -- (-3,0) -- (-3,-2) -- (-4,0) -- cycle;
\fill (0,3) -- (4,3) -- (4,0) -- (3,2) -- (1,2) -- cycle;
\end{scope}

\begin{scope}[yellow!90!red,opacity = 0.1]
\fill (-4,3) -- (0,3) -- (1,2) -- (1,0) -- (-3,0) -- cycle;
\fill (-4,0) -- (-3,-2) -- (-1,-2) -- (0,-3) -- (-4,-3) -- cycle;
\fill (3,2) -- (4,0) -- (4,-3) -- (3, 0) -- cycle;
\end{scope}

\begin{scope}[thick]
\draw (v1) -- (v2) -- (v4) -- (v3) -- (v1) ;
\draw (v1') -- (v2') -- (v4') -- (v3') -- (v1') ;
\end{scope}

\begin{scope}[thick]
\draw (v2) -- (v2');
\draw (v4) -- (0,-3);
\draw (0,3) -- (v4');
\draw (v1) -- (-4,3);
\draw (4,-3) -- (v1');
\draw (v3) -- (-4,0);
\draw (4,0) -- (v3');
\end{scope}

\begin{scope}[-{stealth[scale=2]},gray, very thick]
\draw (-4, -3) -- (-4,3) {};
\draw (4, -3) -- (4,3) {};
\draw (-4, -3) -- (4,-3) {};
\draw (-4, 3) -- (4,3) {};
\end{scope}
\end{tikzpicture}} \quad
		\scalebox{0.8}{\begin{tikzpicture}[scale=0.7]
\tikzset{normal/.style={shape=circle, draw=black, scale=0.4, fill=white}}
\tikzset{type0/.style={shape=circle, draw=red, scale=0.4, fill=white}}
\tikzset{type1/.style={shape=circle, draw=black!30!green, scale=0.4, fill=white}}
\tikzset{type2/.style={shape=circle, draw=black, scale=0.4, fill=white}}
\tikzset{0edge/.style={draw=red, densely dotted}}
\tikzset{1edge/.style={draw=black!40!green, dashed}}
\tikzset{2edge/.style={}}
\tikzset{type12/.style={shape=circle, draw=black!30!green, scale=0.4, fill=white}}

\begin{scope}[-{stealth[scale=2]},gray,very  thick]
\draw (-4, -3) -- (-4,3) {};
\draw (4, -3) -- (4,3) {};
\draw (-4, -3) -- (4,-3) {};
\draw (-4, 3) -- (4,3) {};
\end{scope}

\node[type0] (v1) at (-3,0) {};
\node[type1] (e1) at (-2,0) {};
\node[type0] (v2) at (-1,0) {};
\node[type1] (e2) at (-1,-1) {};
\node[type0] (v3) at (-3,-2) {};
\node[type1] (e3) at (-2,-2) {};
\node[type0] (v4) at (-1,-2) {};
\node[type1] (e4) at (-3,-1) {};
\node[type2] (v5) at (-2,-1) {};

\node[type0] (v1') at (3,0) {};
\node[type1] (e1') at (2,0) {};
\node[type0] (v2') at (1,0) {};
\node[type1] (e2') at (1,1) {};
\node[type0] (v3') at (3,2) {};
\node[type1] (e3') at (2,2) {};
\node[type0] (v4') at (1,2) {};
\node[type1] (e4') at (3,1) {};
\node[type2] (v5') at (2,1) {};

\node[type12] (E) at (0,0) {};
\node[type12] (E1) at (0,-3) {};
\node[type12] (E1') at (0,3) {};
\node[type12] (E2) at (-4,0) {};
\node[type12] (E2') at (4,0) {};
\node[type12] (E3) at (-4,3) {};
\node[type12] (E3') at (4,-3) {};
\node[type12] (E3'') at (-4,-3) {};
\node[type12] (E3''') at (4,3) {};

\node[normal] (f) at (-3,-2.5) {};
\node[normal] (g) at (3,2.5) {};


\begin{scope}[gray, opacity = 0.12]
\fill (4,-3) -- ({18/7},-3) -- (0,0)-- ({-18/7},3) -- (0,3) -- (3,2.5) -- (4,0)-- cycle;
\fill (3,2.5) -- (4,3) --  ({18/7},3) -- cycle;
\fill (-3,-2.5) -- (-4,-3) --  (-4,0) -- cycle;
\fill (0,-3) -- (-3,-2.5) --  ({-18/7},-3) -- cycle;
\end{scope}

\begin{scope}[1edge]
\draw (v3') -- (v5') -- (v4');
\draw (v1') -- (v5')--(v2');
\draw[0edge] (e1') -- (v5') -- (e2');
\draw[0edge] (e4') -- (v5') -- (e3');
\draw (v3) -- (v5) -- (v4);
\draw (v1) -- (v5);
\draw[0edge] (e1) -- (v5) -- (e2);
\draw[0edge] (e4) -- (v5) -- (e3);
\draw[2edge] (e1) -- (v1) -- (e4)--(v3) -- (e3) -- (v4) -- (e2) -- (v2) -- (e1);
\draw (v2) -- (v5);
\draw[2edge]  (e1') -- (v1') -- (e4')--(v3') -- (e3') -- (v4') -- (e2') -- (v2') -- (e1');

\draw[2edge] (v2) -- (E) -- (v2');
\draw[2edge] (v4') -- (E1');
\draw[2edge] (v4) -- (E1);
\draw[2edge] (v3) -- (E2);
\draw[2edge] (v3') -- (E2');
\draw[2edge] (v1) -- (E3);
\draw[2edge] (v1') -- (E3');

\draw[0edge] (e3) -- (f) ;

\draw[0edge] (e3') -- (g) ;

\draw[0edge] (-4,-0.75) -- (f);
\draw[0edge] (4,-0.75) -- (e4');
\draw[0edge] (-{20/7},-3) -- (f);
\draw[0edge] ({-20/7},3) -- (e1);
\draw[0edge] (-2.2,-3) -- (f);
\draw[0edge] (-2.2,3) -- (e2');

\draw[0edge] (4,0.75) -- (g);
\draw[0edge] (-4,0.75) -- (e4);
\draw[0edge] ({20/7},3) -- (g);
\draw[0edge] ({20/7},-3) -- (e1');
\draw[0edge] (2.2,3) -- (g);
\draw[0edge] (2.2,-3) -- (e2);

\draw (v4) -- (f) -- (v3);
\draw (-4, -0.25) -- (f) -- (-4,-1.25);
\draw (4, -0.25) -- (v3');
\draw (4, -1.25) -- (v1');
\draw (-3,-3) -- (f) -- (-{19/7},-3);
\draw (-3,3) -- (v1);
\draw (-{19/7},3) -- (v2);
\draw (-{5/3},-3) -- (f) -- (-{17/7},-3);
\draw (-{5/3},3) -- (v4');
\draw (-{17/7},3) -- (v2');

\draw (v4') -- (g) -- (v3');
\draw (4, 0.25) -- (g) -- (4,1.25);
\draw (-4, 0.25) -- (v3);
\draw (-4, 1.25) -- (v1);
\draw (3,3) -- (g) -- ({19/7},3);
\draw (3,-3) -- (v1');
\draw ({19/7},-3) -- (v2');
\draw ({5/3},3) -- (g) -- ({17/7},3);
\draw ({5/3},-3) -- (v4);
\draw ({17/7},-3) -- (v2);

\end{scope}

\begin{scope}[0edge, very thick]
\draw ({18/7},3) -- (g);
\draw ({18/7},-3) -- (E);
\draw (-{18/7},-3) -- (f);
\draw ({-18/7},3) -- (E);

\draw (f) -- (E2);
\draw (E1) -- (f) -- (E3'');
\draw (E1') -- (g) -- (E3''');
\draw (g) -- (E2');
\end{scope}

\end{tikzpicture}} \quad
		\scalebox{0.8}{\begin{tikzpicture}[scale=0.7]
\tikzset{normal/.style={shape=circle, draw=black, scale=0.4,fill=black}}
\tikzset{greenface/.style={shape=circle, draw=yellow!90!red, scale=0.4,fill=yellow!90!red}}
\tikzset{redface/.style={shape=circle, draw=blue, scale=0.4,fill=blue}}

\node[normal] (v1) at (-3,0) {};
\node[normal] (v2) at (-1,0) {};
\node[normal] (v3) at (-3,-2) {};
\node[normal] (v4) at (-1,-2) {};
\node[normal] (v5) at (-2,-1) {};

\node[normal] (v1') at (3,0) {};
\node[normal] (v2') at (1,0) {};
\node[normal] (v3') at (3,2) {};
\node[normal] (v4') at (1,2) {};
\node[normal] (v5') at (2,1) {};



\node[greenface] (f) at (-3,-2.5) {};
\node[redface] (g) at (3,2.5) {};

\begin{scope}[blue,opacity = 0.05]
\fill (3,0) -- (4,-3) -- (0,-3) -- (-1,-2) -- (-1,0) -- cycle;
\fill (-4,3) -- (-3,0) -- (-3,-2) -- (-4,0) -- cycle;
\fill (0,3) -- (4,3) -- (4,0) -- (3,2) -- (1,2) -- cycle;
\end{scope}

\begin{scope}[yellow!90!red,opacity = 0.05]
\fill (-4,3) -- (0,3) -- (1,2) -- (1,0) -- (-3,0) -- cycle;
\fill (-4,0) -- (-3,-2) -- (-1,-2) -- (0,-3) -- (-4,-3) -- cycle;
\fill (3,2) -- (4,0) -- (4,-3) -- (3, 0) -- cycle;
\end{scope}

\begin{scope}[thick]
\draw (v3') -- (v5') -- (v4');
\draw (v3) -- (v5) -- (v4);
\draw (v2) -- (v5) -- (v1);
\draw (v2') -- (v5') -- (v1');
\end{scope}

\begin{scope}[thick,yellow!90!red]
\draw (v4) -- (f) -- (v3);
\draw (-4, -0.25) -- (f) -- (-4,-1.25);
\draw (4, -0.25) -- (v3');
\draw (4, -1.25) -- (v1');
\draw (-3,-3) -- (f) -- (-{19/7},-3);
\draw (-3,3) -- (v1);
\draw (-{19/7},3) -- (v2);
\draw (-{5/3},-3) -- (f) -- (-{17/7},-3);
\draw (-{5/3},3) -- (v4');
\draw (-{17/7},3) -- (v2');
\end{scope}

\begin{scope}[thick,blue]
\draw (v4') -- (g) -- (v3');
\draw (4, 0.25) -- (g) -- (4,1.25);
\draw (-4, 0.25) -- (v3);
\draw (-4, 1.25) -- (v1);
\draw (3,3) -- (g) -- ({19/7},3);
\draw (3,-3) -- (v1');
\draw ({19/7},-3) -- (v2');
\draw ({5/3},3) -- (g) -- ({17/7},3);
\draw ({5/3},-3) -- (v4);
\draw ({17/7},-3) -- (v2);
\end{scope}

\begin{scope}[lightgray]
\draw (v1) -- (v2) -- (v4) -- (v3) -- (v1) ;
\draw (v1') -- (v2') -- (v4') -- (v3') -- (v1');
\end{scope}

\begin{scope}[lightgray]
\draw (v2) -- (v2');
\draw (v4) -- (0,-3);
\draw (0,3) -- (v4');
\draw (v1) -- (-4,3);
\draw (4,-3) -- (v1');
\draw (v3) -- (-4,0);
\draw (4,0) -- (v3');
\end{scope}

\begin{scope}[-{stealth[scale=2]},gray,very  thick]
\draw (-4, -3) -- (-4,3) {};
\draw (4, -3) -- (4,3) {};
\draw (-4, -3) -- (4,-3) {};
\draw (-4, 3) -- (4,3) {};
\end{scope}
\end{tikzpicture}}
		\caption{The figure on the left-hand side shows a 3-connected graph embedded on the torus. The faces that correspond to the 2-cut are colored blue and yellow. The middle figure shows the barycentric subdivision of the graph, with a subgraph that has two simple faces as explained in the proof of Theorem~\ref{thm:edge-breaking_breaks}. One of them is colored gray. The last figure shows the (edge-breaking) operation Join applied to the graph. The blue and the yellow vertex form a 2-cut, their incident edges are also drawn in blue and yellow. }
		\label{fig:cube_counter}
	\end{figure}
	Let $G$ be the graph on the left-hand side of Figure \ref{fig:cube_counter}. It is a 3-connected embedding of genus 1. Let $O$ be any edge-breaking operation. Then $O(G)$ will have a vertex $f$ corresponding to the blue face and a vertex $g$ corresponding to the yellow face. 
	In the second drawing of Figure~\ref{fig:cube_counter} some 0-edges of the barycentric subdivision of $G$ are highlighted. By the definition of an edge-breaking operation these edges appear in $B_{O(G)}$, although they may have a different type. The highlighted edges form a subgraph of $B_{O(G)}$ that has two simple faces and only two vertices -- $f$ and $g$ -- of type 0. As there is at least one vertex of type 0 in the interior of each face, this implies that $f$ and $g$ form a 2-cut of $O(G)$. An infinite number of examples of arbitrary genus $> 0$ can be constructed from $G$ by adding more vertices and edges in the interior of the faces of $G$ of length four (the white faces in Figure~\ref{fig:cube_counter}). As this does not change the faces $f$ or $g$ the arguments still hold.
\end{proof}

\begin{theorem}\label{thm:main_2}
	Given a $c3$-lopsp-operation $O$, the embedded (multi)graph $O(G)$ is 3-connected for any simple 3-connected embedded graph $G$ if and only if $O$ is edge-preserving.
\end{theorem}
\begin{proof}
	This follows from Theorem~\ref{thm:main_1} and Theorem~\ref{thm:edge-breaking_breaks}.
\end{proof}

Theorem~\ref{thm:main_2} characterises exactly which lopsp-operations may break 3-connectivity. In some situations however, one may only be interested in simple graphs, and in Theorem~\ref{thm:main_2} the result of applying an operation can be a multigraph. Theorem~\ref{thm:simple} proves that edge-breaking operations of type 1 preserve 3-connectivity under the extra condition that $O(G)$ is simple, and that this is not true for edge-breaking operations of type 2. To prove the second statement Corollary~\ref{cor:simple_result} is used. 

\begin{corollary}\label{cor:simple_result}
Let $O$ be a $c3$-lopsp-operation and let $G$ be a simple 3-connected embedded graph such that $B_G$ is simple. If $O$ is not edge-breaking of type 1, then $O(G)$ is simple.
\end{corollary}
\begin{proof}
	We can assume that $O$ is not the identity operation or Dual: For the identity operation the statement is obvious, and Dual is an edge-breaking operation of type 1.
	
	Assume that $O(G)$ is not simple. Then there is a type-2 cycle $c$ in $B_{O(G)}$ of length two or length four. A type-2 cycle of length two or four cannot be trivial, so $c$ is non-trivial. As $B_G$ is simple, it follows from Lemma~\ref{lem:4cycle_positions} that two non-adjacent vertices of $c$ are 2-points and the other 2 vertices are on different 2-sides. If the vertices of $c$ on the 2-sides are 1-points, then $O$ is an edge-breaking operation of type 1. If they are not then by symmetry it follows that there is a type-2 cycle of length four in the diamond corresponding to the chosen cut-path. Since a type-2 cycle cannot be trivial this is a contradiction with Lemma~\ref{lem:c3char}.	
\end{proof}

\begin{theorem}\label{thm:simple}
	Let $O$ be an edge-breaking $c3$-lopsp-operation different from Dual. 
	\begin{itemize}
		\item \textbf{If $O$ is of type 1:} For any simple 3-connected embedded graph $G$: If $O(G)$ is simple then $O(G)$ is 3-connected.
		\item \textbf{If $O$ is of type 2:} There exist simple 3-connected embedded graphs $G$ such that $O(G)$ is simple and not 3-connected. 
	\end{itemize}
\end{theorem}
\begin{proof}
	\begin{itemize}
		\item \textbf{$O$ is of type 1:} By the definition of an edge-breaking operation of type 1, the dual $G^*$ is a subgraph of $O(G)$. It follows that if $O(G)$ is simple, then $G^*$ is also simple. Now Theorem~\ref{thm:main_1} implies that $O(G)$ is 3-connected.
		
		\item \textbf{$O$ is of type 2:} By Corollary~\ref{cor:simple_result} it suffices to prove that there are simple 3-connected embedded graphs $G$ such that $B_G$ is simple and $O(G)$ is not 3-connected for any edge-breaking $O$. The graph on the left-hand side in Figure~\ref{fig:cube_counter} is such a graph. 
	\end{itemize}
\end{proof}

\begin{table}
	\centering
	\begin{tabular}{|c|cccc|}
		\hline 
		&  \multicolumn{3}{c|}{Edge-breaking} & Edge-  \\ 
		\cline{2-4}
		& Dual & Type 2 & \multicolumn{1}{c|}{Type 1$'$} & preserving \\ 
		\hline 
		$G$ plane& \textcolor{black!30!green}{Yes} & \textcolor{black!30!green}{Yes} &\textcolor{black!30!green}{Yes} & \textcolor{black!30!green}{Yes} \\ 
		$G$ polyhedral& \textcolor{black!30!green}{Yes} & \textcolor{black!30!green}{Yes} &\textcolor{black!30!green}{Yes} & \textcolor{black!30!green}{Yes}\\ 
		$G^*$ simple& \textcolor{red}{No} & \textcolor{black!30!green}{Yes} & \textcolor{black!30!green}{Yes} & \textcolor{black!30!green}{Yes} \\ 
		$O(G)$ simple& \textcolor{red}{No} & \textcolor{red}{No} &\textcolor{black!30!green}{Yes} & \textcolor{black!30!green}{Yes} \\ 
		General& \textcolor{red}{No} & \textcolor{red}{No} &  \textcolor{red}{No}&  \textcolor{black!30!green}{Yes}\\
		\hline 
	\end{tabular} 
	\caption{In column $A$ and row $B$ this table answers the question `If $O$ is a $c3$-lopsp-operation of class $A$, is $O(G)$ 3-connected for every simple 3-connected graph $G$ that satisfies condition $B$?'.}
	\label{tab:summary}
\end{table}

In Table~\ref{tab:summary} the effect of different types of lopsp-operations on 3-connectivity is summarised. 
Every column represents a different class of operations, and every row represents a different condition on the embedded graph $G$. Note that a `No' does not mean that the result of such an operation cannot be 3-connected. It simply means that for every lopsp-operation of the given class, there exist 3-connected graphs satisfying the condition such that the result of applying the operation is not 3-connected. In the table, Dual is separated from the other operations of type 1. We say that these other operations are of type $1'$. This way, every lopsp-operation is in exactly one of the given classes. For the conditions on the graphs this is not the case. An embedded graph can satisfy more than one of the conditions.

The first row shows that for all plane embedded graphs, all lopsp-operations preserve 3-connectivity. This was proved for lsp-operations in \cite{brinkmann2017comparing}. In \cite{brinkmann2021local} it was proved for lopsp-operations in a more general context, also including the results in the second row. The results for Dual can be found in \cite{bokal2022connectivity} and \cite{mohar1997face}. All other cells in the table show new results. The last row and column follow from Theorem~\ref{thm:main_2}. In the last row there are no restrictions on the graphs, except that they must be 3-connected. Edge-preserving operations are the only operations that always preserve 3-connectivity, so they are the only class with a `Yes' in that row, and consequently also all other rows. If we demand that the result of the operation is simple, then edge-breaking operations of type $1'$ also preserve 3-connectivity, but operations of type 2 and Dual do not. This follows from Theorem~\ref{thm:simple}. For graphs with a simple dual, it is even the case that Dual is the only operation that may destroy 3-connectivity, which follows from Theorem~\ref{thm:main_1}.

\section{Conclusion and extensions}
As the general approach to lopsp-operations was described as recently as in 2017
\cite{brinkmann2017comparing}, all earlier research on symmetry-preserving operations focused on specific
operations, and there were no general results about the whole class.
In this text we discussed the question `Which lopsp-operations preserve 3-connectivity in simple embedded graphs?'. The answer to this question for different types of graphs and operations is summarised in Table \ref{tab:summary}. 

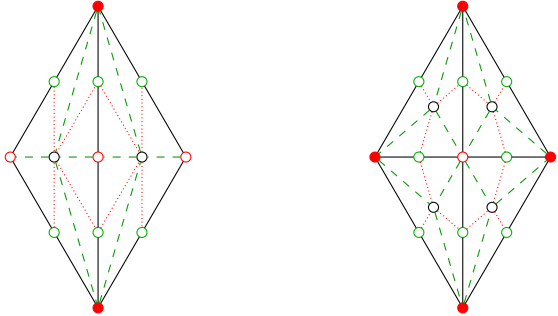
\begin{figure}		
	\centering
	\scalebox{1}{\begin{tikzpicture}
\tikzset{type0/.style={shape=circle, draw=red, scale=0.4, fill=white}}
\tikzset{type1/.style={shape=circle, draw=black!30!green, scale=0.4, fill=white}}
\tikzset{type2/.style={shape=circle, draw=black, scale=0.4, fill=white}}
\tikzset{0edge/.style={draw=red, densely dotted}}
\tikzset{1edge/.style={draw=black!40!green, dashed}}
\tikzset{2edge/.style={}}

\node[type0] (v0) at (-{2*tan(30},0) {};
\node[type0] (v0') at ({2*tan(30},0) {};
\node[type0] (v1) at (0,0) {};
\node[type0, fill=red] (v2') at (0,-2) {};
\node[type0, fill=red] (v2) at (0,2) {};

\node[type2] (f') at ({1*tan(30},0) {};
\node[type2] (f) at ({-1*tan(30},0) {};
\node[type1] (e1) at ({-tan(30},  1) {};
\node[type1] (e2) at ({tan(30},  1) {};
\node[type1] (e3) at ({tan(30},  -1) {};
\node[type1] (e4) at ({-tan(30},  -1) {};
\node[type1] (E) at (0,  1) {};
\node[type1] (E') at (0,  -1) {};

\begin{scope}[2edge]
\draw (v0) -- (e1) --(v2)--(e2) -- (v0') -- (e3) -- (v2') -- (e4) -- (v0);
\draw (v2) -- (E) -- (v1) -- (E') -- (v2');
\end{scope}

\begin{scope}[1edge]
\draw (v0') -- (f') -- (v1) -- (f) -- (v0);
\draw (v2) -- (f) -- (v2') --(f') -- (v2); 

\end{scope}

\begin{scope}[0edge]
\draw (e1) -- (f) -- (E) -- (f') -- (e2);
\draw (e4) -- (f) -- (E') -- (f') -- (e3);
\end{scope}

\end{tikzpicture}} \qquad
	\qquad
	\qquad
	\scalebox{1}{\begin{tikzpicture}
\tikzset{type0/.style={shape=circle, draw=red, scale=0.4, fill=white}}
\tikzset{type1/.style={shape=circle, draw=black!30!green, scale=0.4, fill=white}}
\tikzset{type2/.style={shape=circle, draw=black, scale=0.4, fill=white}}
\tikzset{0edge/.style={draw=red, densely dotted}}
\tikzset{1edge/.style={draw=black!40!green, dashed}}
\tikzset{2edge/.style={}}

\node[type0, fill=red] (v0) at (-{2*tan(30},0) {};
\node[type0, fill=red] (v0') at ({2*tan(30},0) {};
\node[type0] (v1) at (0,0) {};
\node[type0, fill=red] (v2') at (0,-2) {};
\node[type0, fill=red] (v2) at (0,2) {};

\node[type1] (E'') at ({1*tan(30},0) {};
\node[type1] (E''') at ({-1*tan(30},0) {};
\node[type1] (e1) at ({-tan(30},  1) {};
\node[type1] (e2) at ({tan(30},  1) {};
\node[type1] (e3) at ({tan(30},  -1) {};
\node[type1] (e4) at ({-tan(30},  -1) {};
\node[type1] (E) at (0,  1) {};
\node[type1] (E') at (0,  -1) {};
\node[type2] (f1) at ({-2*tan(30)/3},  2/3) {};
\node[type2] (f2) at ({2*tan(30)/3},  2/3) {};
\node[type2] (f3) at ({2*tan(30)/3},  -2/3) {};
\node[type2] (f4) at ({-2*tan(30)/3},  -2/3) {};

\begin{scope}[2edge]
\draw (v0) -- (e1) --(v2)--(e2) -- (v0') -- (e3) -- (v2') -- (e4) -- (v0);
\draw (v2) -- (E) -- (v1) -- (E') -- (v2');
\draw (v0') -- (E'') -- (v1) -- (E''') -- (v0);
\end{scope}

\begin{scope}[1edge]
\draw (v2) -- (f1) -- (v0) -- (f4) -- (v2') -- (f3) -- (v0') -- (f2) -- (v2)
		(f1) -- (v1) -- (f4)
		(f2) -- (v1) -- (f3);

\end{scope}

\begin{scope}[0edge]
\draw (e1) -- (f1) -- (E) -- (f2) -- (e2)
	     (e4) -- (f4) -- (E') -- (f3) -- (e3)
	     (f1) -- (E''') -- (f4)
	     (f2) -- (E'') -- (f3);
\end{scope}

\end{tikzpicture}} \qquad
	\caption{Diamonds of two lopsp-operations. The left operation can be associated with a 2-connected tiling and $O(G)$ has a 1-cut if $G^*$ has a loop. The right operation can be associated with a 4-connected tiling and $O(G)$ has a 3-cut if $G^*$ has a loop. The vertices that induce the cut are filled.}
	\label{fig:examples_2_4}
\end{figure}

We focused on 3-connectivity here because it is the most natural concept in this context: Lsp- and lopsp-operations stem from the study of polyhedra, which are 3-connected plane embedded graphs. In fact, lsp-operations were first defined as what we described here as $c3$-operations. All well-known operations come from 3-connected tilings. Moreover, there are many results that are only true for 3-connected graphs. At first glance, the proof of Theorem~\ref{thm:main_1} seems to work for $k$-connected graphs for $k\neq3$ as well. Lemma~\ref{lem:2_broken} can be adapted for this situation with a small adaptation of the definition of edge-breaking operations. Lemma~\ref{lem:diff_v0sets} on the other hand is not true for $k\neq 3$ unless more complex additional constraints are imposed on the operations. Figure~\ref{fig:examples_2_4} shows examples of edge-preserving operations that do not preserve 2- or 4-connectivity. Excluding these situations would require much more complicated definitions of `connectivity-preserving' operations. 

As multigraphs appear naturally when applying lopsp-operations, an extension of these results could be to apply $c3$-lopsp-operations to 3-connected graphs that are not simple. This requires many modifications of the definitions and proofs. If there are loops operations such as the second one in Figure~\ref{fig:examples_2_4} introduce 2-cuts if a loop has the same face on both sides. It seems that multigraphs with loops lose their connectivity more easily than multigraphs with only multiple edges, but even for multigraphs without loops the results here are not valid. In Figure~\ref{fig:multigraph} an example is given of a 3-connected multigraph without loops that is not 3-connected after applying an edge-preserving lopsp-operation.

\begin{figure}	
	\centering
	\scalebox{0.8}{\begin{tikzpicture}[scale=0.7]
\tikzset{normal/.style={shape=circle, draw=black, scale=0.4,fill=black}}

\node[normal] (v1) at (-2,0) {};
\node[normal] (v2) at (0,-0.75) {};
\node[normal] (v3) at (2,0) {};
\node[normal] (v4) at (0,0.75) {};

\begin{scope}[thick]
\draw (v1) -- (v2) -- (v3) -- (v4) -- (v1);
\draw (v1) -- (-4,0);
\draw (v2) -- (v4);
\draw (v3) -- (4,0);
\draw (v1) edge[in =110, out=70, looseness=1.2] (v3);
\draw (v1) edge[in =-110, out=-70, looseness=1.2] (v3);
\draw (v1) edge[out = 90, in = -90] (0,3);
\draw (0,-3) edge[out = 90, in = -90] (v3);
\end{scope}

\begin{scope}[green,opacity = 0.1]
\end{scope}

\begin{scope}[red,opacity = 0.1]
\end{scope}

\begin{scope}[-{stealth[scale=2]},gray, very thick]
\draw (-4, -3) -- (-4,3) {};
\draw (4, -3) -- (4,3) {};
\draw (-4, -3) -- (4,-3) {};
\draw (-4, 3) -- (4,3) {};
\end{scope}

\end{tikzpicture}} \qquad \qquad
	\scalebox{0.8}{\begin{tikzpicture}[scale=0.7]
\tikzset{normal/.style={shape=circle, draw=black, scale=0.4,fill=black}}
\tikzset{bluevertex/.style={shape=circle, draw=blue, scale=0.4,fill=blue}}
\tikzset{orangevertex/.style={shape=circle, draw=yellow!90!red, scale=0.4,fill=yellow!90!red}}

\begin{scope}[-{stealth[scale=2]},gray, very thick]
\draw (-4, -3) -- (-4,3) {};
\draw (4, -3) -- (4,3) {};
\draw (-4, -3) -- (4,-3) {};
\draw (-4, 3) -- (4,3) {};
\end{scope}

\node[bluevertex] (v1) at (-2,0) {};
\node[normal] (v2) at (0,-0.75) {};
\node[orangevertex] (v3) at (2,0) {};
\node[normal] (v4) at (0,0.75) {};
\node[normal] (f1) at (-0.65,0) {};
\node[normal] (f2) at (0.65,0) {};
\node[normal] (f3) at (0,1.1) {};
\node[normal] (f4) at (0,-1.1) {};
\node[normal] (x1) at (-4,3) {};
\node[normal] (x2) at (4,3) {};
\node[normal] (x3) at (4,-3) {};
\node[normal] (x4) at (-4,-3) {};

\begin{scope}[thick]
\draw[blue] (v4) --(v1) -- (v2);
\draw[yellow!90!red] (v2) -- (v3) -- (v4);
\draw[yellow!90!red!70!blue] (v1) -- (-4,0);
\draw[very thick] (v2) -- (v4);
\draw[yellow!90!red!70!blue] (v3) -- (4,0);
\draw[yellow!90!red!70!blue] (v1) edge[in =110, out=70, looseness=1.2] (v3);
\draw[yellow!90!red!70!blue] (v1) edge[in =-110, out=-70, looseness=1.2] (v3);
\draw[yellow!90!red!70!blue] (v1) edge[out = 90, in = -90] (0,3);
\draw[yellow!90!red!70!blue] (0,-3) edge[out = 90, in = -90] (v3);
\end{scope}

\begin{scope}[thick]
\draw[very thick] (f1) -- (v2) --(f2);
\draw[yellow!90!red] (f2)-- (v3);
\draw[ very thick] (f1) -- (v4) -- (f2);
\draw[very thick] (v2) -- (f4)
	     (v4) -- (f3);
\draw[blue] (x4) -- (v1) -- (x1);
\draw[blue] (f1) -- (v1);
\draw[yellow!90!red] (x2) -- (v3) -- (x3);
\draw (v1) edge[out=45, in=180, looseness=0.75, blue] (f3);
\draw (v1) edge[out=-45, in=180, looseness=0.75, blue] (f4);
\draw (v3) edge[out=135, in=0, looseness=0.75, yellow!90!red] (f3);
\draw (v3) edge[out=-135, in=0, looseness=0.75, yellow!90!red] (f4);
\draw (v1) edge[out=80, in=-150, looseness=0.8, blue] (x2);
\draw (v3) edge[out=-100, in=30, looseness=0.8, yellow!90!red] (x4);
\end{scope}

\end{tikzpicture}} 
	\caption{The left figure shows an embedding of a 3-connected multigraph on the torus. On the right the result of applying the edge-preserving lopsp-operation Kis -- also known as Stellation -- is shown. The vertices of the 2-cut and their incident edges are colored blue, yellow and light brown depending on which vertex or vertices they are incident with.}
	\label{fig:multigraph}
\end{figure}
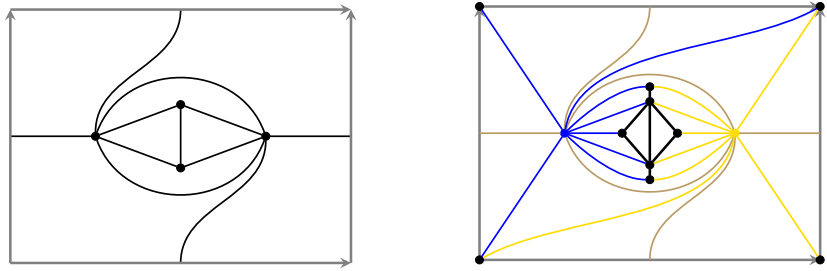

Our results show that the new description for symmetry-preserving operations can be used to prove very general results about all of these operations at the same time. It is our hope that this method will also prove successful to determine the effect of lsp- or lopsp-operations on other graph invariants such as hamiltonicity or the minimum genus.

\section*{Acknowledgements}
I would like to thank Gunnar Brinkmann for his advice and helpful discussions.

\newpage

		\bibliographystyle{amseigen}
		\bibliography{bibliography}
\end{document}